\pdfoutput=1
\documentclass[11pt,a4paper,oneside]{article}
\usepackage{ifthen}
\newboolean{journal}
\setboolean{journal}{false}  
\usepackage[DIV=12]{typearea}
\usepackage[small, bf]{titlesec}
\usepackage{ellipsis}
\usepackage[noblocks]{authblk}

\usepackage{amssymb,amsmath,amsfonts,amsthm}
\usepackage[english]{babel}
  \usepackage[T1]{fontenc}
  \usepackage[utf8]{inputenc}
\usepackage{enumerate,expdlist,tikz,float}
  \usetikzlibrary{patterns}
\usepackage{xcolor}
\colorlet{VeryRed}{red!90!black!80}   
\renewcommand{\epsilon}{\varepsilon}
\newcommand{\euler}{\mathrm{e}}
\newcommand{\rout}{R}
\newcommand{\R}{R}
\newcommand{\T}{\mathrm{T}}
\newcommand{\rin}{r}
\newcommand{\Eins}{\chi}
\newcommand{\sfuc}{\mathrm{sfuc}}
\newcommand{\drm}{\mathrm{d}}
\newcommand{\RR}{\mathbb{R}}
\newcommand{\CC}{\mathbb{C}}
\newcommand{\NN}{\mathbb{N}}	
\newcommand{\ZZ}{\mathbb{Z}}
\newcommand{\PP}{\mathbb{P}}
\newcommand{\EE}{\mathbb{E}}

\newcommand{\supp}{\operatorname{supp}}

\newcommand{\cut}{\Theta}
\newcommand{\x}{z}
\newcommand{\au}{\alpha_1}
\newcommand{\ao}{\alpha_2}
\newcommand{\bu}{\beta_1}
\newcommand{\bo}{\beta_2}
\newcommand{\BIGOP}[1]{\mathop{\mathchoice%
{\raise-0.22em\hbox{\huge $#1$}}%
{\raise-0.05em\hbox{\Large $#1$}}{\hbox{\large $#1$}}{#1}}}
\newcommand{\BIGboxplus}{\mathop{\mathchoice%
{\raise-0.35em\hbox{\huge $\boxplus$}}%
{\raise-0.15em\hbox{\Large $\boxplus$}}{\hbox{\large $\boxplus$}}{\boxplus}}}
\newcommand{\bigtimes}{\BIGOP{\times}}
\DeclareMathOperator{\funs}{s}

\newtheorem{theorem}{Theorem}[section]
\newtheorem{lemma}[theorem]{Lemma}
\newtheorem{proposition}[theorem]{Proposition}
\newtheorem{corollary}[theorem]{Corollary}
\theoremstyle{definition}
\newtheorem{definition}[theorem]{Definition}
\theoremstyle{remark}
\newtheorem{remark}[theorem]{Remark}

\usepackage[textsize=tiny]{todonotes}
\usepackage{cancel}

\begin{document}
%
% %%%%%%%%%%%%%%%%%%%%%%%%%%%%%%%%%%%%%%%%%%%%%%%%%%%%%%%%%%%%%%%%%%%%%%
% %---------------------------------------------------------------------
%
%         T I T L E
%
% %---------------------------------------------------------------------
% %%%%%%%%%%%%%%%%%%%%%%%%%%%%%%%%%%%%%%%%%%%%%%%%%%%%%%%%%%%%%%%%%%%%%%
% %
%
\title{Scale-free unique continuation principle, eigenvalue lifting and Wegner estimates for random Schr\"odinger operators}

\author{Ivica Naki\'c}
\affil{University of Zagreb, Department of Mathematics, Croatia}
\author{Matthias T\"aufer}
\author{Martin Tautenhahn}
\author{Ivan Veseli\'c}
\affil{Technische Universit\"at Chemnitz, Fakult\"at f\"ur Mathematik, Germany}
\date{\vspace{-3em}}
\maketitle
\begin{abstract}
We prove a scale-free, quantitative unique continuation principle for 
functions in the range of the spectral projector $\chi_{(-\infty,E]}(H_L)$ 
of a Schr\"odinger operator $H_L$ on a cube of side $L\in \NN$, with bounded potential.
Such estimates are also called, depending on the context, uncertainty principles, observability estimates, or spectral inequalities.
We apply it to (i) prove a Wegner estimate for random Schr\"odinger operators with non-linear parameter-dependence and to (ii) exhibit the dependence of the control cost on geometric model parameters for the heat equation in a multi-scale domain.
\end{abstract}
\ifthenelse{\boolean{journal}}{}{\tableofcontents}%

%%%%%%%%%%%%
%%%
%%% Introduction
%%%
%%%%%%%%%%%%

\section{Introduction}

We prove a \emph{quantitative unique continuation} inequality for 
functions in the range of the spectral projector $\chi_{(-\infty,E]}(H_L)$ 
of a Schr\"odinger  operator $H_L$ on a cube of side $L\in \NN$. It has been announced in \cite{NakicTTV-15}.
Depending on the area of mathematics and the context such estimates have various names: quantitative unique continuation principle (UCP),
uncertainty principles, spectral inequalities, observability or sampling estimates, or bounds on the vanishing order.
If the observability or sampling set  respects in a certain way the underlying lattice structure, our estimate is independent of $L$; for this
reason we call it \emph{scale-free}.
For our applications it is crucial to exhibit explicitly the dependence of the {quantitative unique continuation} inequality
on the model parameters. 
\par
A key motivation to study scale-free quantitative unique continuation estimates comes from 
the theory of random Schr\"odinger operators, in particular eigenvalue lifting estimates, Wegner bounds, and the continuity of the integrated density of states. 
(We defer precise definitions to \S \ref{s:results}.)
In fact, there is quite a number of previous papers which have derived a scale-free UCP 
and eigenvalue lifting estimates under special assumptions.
Naturally, the first situation to be considered was the case where
the Schr\"odinger operator is the pure Laplacian  $H=-\Delta$, i.e. the background  potential $V$ vanishes identically.
For instance, \cite{Kirsch-96} derives a UCP which is valid for energies in an interval at zero, i.e.~the bottom of the spectrum, 
if one has a periodic arrangement of sampling sets. 
The proof uses detailed information about hitting probabilities of Brownian motion paths, and is in sense related to Harnack inequalities.
A very elementary approach to eigenvalue lifting estimates is provided by the spatial averaging trick, used in \cite{BourgainK-05}
and \cite{GerminetHK-07} in periodic situations,  and extended to non-periodic situations in \cite{Germinet-08}. 
It is applicable to energies near zero.
A different approach for eigenvalue lifting was derived in \cite{BoutetdeMonvelNSS-06}.
In \cite{BoutetdeMonvelLS-11}  it was shown how one can conclude an uncertainty principle at low energies
based on an eigenvalue lifting estimate.
Related results have been derived for energies near spectral edges in \cite{KirschSS-98a} and \cite{CombesHN-01} using resolvent comparison.
In one space dimension eigenvalue lifting  results and Wegner estimates have been proven in
\cite{Veselic-96}, \cite{KirschV-02b}. There a periodic arrangement of the sampling set is assumed.
The proof carries over to the case of non-periodic arrangements verbatim, which has been used in the context of quantum graphs in \cite{HelmV-07}.
In the case that both the deterministic background potential and the sampling set are periodic,
an uncertainty principle and a Wegner estimate, which are valid for arbitrary bounded energy regions, have been proven in \cite{CombesHK-03,CombesHK-07}.
These papers make use of Floquet theory, hence they are a priori
restricted to periodic background potentials as well as periodic sampling sets.
An alternative proof for the result in \cite{CombesHK-07}, with more explicit control of constants, has been worked out in \cite{GerminetK-13}.
The case where the background potential is periodic but the impurities need not be periodically arranged has been considered
in \cite{BoutetdeMonvelNSS-06} and \cite{Germinet-08} for low energies.
Our main theorem unifies and generalizes all the results mentioned so far 
and  makes the dependence on the model parameters explicit.
Indeed, our scale-free unique continuation principle answers positively a question asked in \cite{RojasMolinaV-13}.
A partial answer was given already in \cite{Klein-13}.
While \cite{RojasMolinaV-13} concerns the case of a single eigenfunctions, \cite{Klein-13} treats linear combinations of eigenfunctions
corresponding to very close eigenvalues.
For a broader discussion we refer to the summer school notes \cite{TaeuferTV-16}.
\par
A second application of our scale-free UCP is in the control theory of the heat equation.
Here one asks whether one can drive a given  initial state to a desired state with a control function 
living in a specified subset, and what  the minimal $L^2$-norm of the control function (called control cost) is.
Recently, the search for optimal placement of the control set and the dependence of the control cost
on geometric features of this set has received much attention, see e.g.~\cite{PrivatTZ-15a,PrivatTZ-15b}. 
Our scale-free UCP gives an explicit estimate of the control cost w.r.t.\ the model parameters in multi-scale domains.
\par
Our proof of the scale-free unique continuation estimate uses two Carleman and nested interpolation estimates, 
an idea used before e.g.\ in \cite{LebeauR-95,JerisonL-99}. To obtain explicit estimates we need explicit weight functions.
The first Carleman estimate includes a boundary term and uses a parabolic weight function as proposed in \cite{JerisonL-99}.
The second Carleman estimate is similar to the ones in \cite{EscauriazaV-03,BourgainK-05}.
However, none of the two is quite sufficient for our purposes, so we use a variant developed in \cite{NakicRT-15}, see also Appendix \S \ref{app:carleman}.
Moreover, typically the diameter of the ambient manifold enters in the Carleman estimate.
In our case it grows unboundedly in $L$, hence the UCP constants would become worse and worse. Thus, to eliminate the $L$-dependence we have to use 
techniques developed in the context of random Schr\"odinger operators to accommodate for the multi-scale structure of the underlying domain and sampling set.  
\par
In the next section we state our main results,
\S \ref{s:proof-sfuc}
is devoted to the proof of the scale-free unique continuation principle,
\S \ref{s:proof-Wegner}
to proofs concerning random Schr\"odinger  operators,
\S \ref{s:proof-observability}
to the observability estimate of the control equation,
while certain technical aspects are deferred to the appendix.
%%%%%%%%%%%%
%%%
%%% Results
%%%
%%%%%%%%%%%%
\section{Results}
\label{s:results}
\subsection{Scale-free unique continuation and eigenvalue lifting}
Let $d \in \NN$. For $L > 0$ we denote by $\Lambda_L = (-L/2 , L/2)^d \subset \RR^d$ the cube with side length $L$, and by $\Delta_L$ the Laplace operator on $L^2 (\Lambda_L)$ with Dirichlet, Neumann or periodic boundary conditions.
Moreover, for a measurable and bounded $V : \RR^d \to \RR$ we denote by $V_L : \Lambda_L \to \RR$ its restriction to $\Lambda_L$ given by $V_L (x)  = V (x)$ for $x \in \Lambda_L$, and by
\[
H_L = -\Delta_L + V_L \quad \text{on} \quad L^2 (\Lambda_L) 
\]
the corresponding Schr\"odinger operator.
Note that $H_L$ has purely discrete spectrum. For $x \in \RR^d$ and $r > 0$ we denote by $B(x , r)$ the ball with center $x$ and radius $r$ with respect to Euclidean norm. If the ball is centered at zero we write $B (r) = B (0,r)$.
\begin{definition}
 Let $G > 0$ and $\delta > 0$. We say that a sequence $z_j \in \RR^d$, $j \in (G \ZZ)^d$ is \emph{$(G,\delta)$-equidistributed}, if
 \[
  \forall j \in (G \ZZ)^d \colon \quad  B(\x_j , \delta) \subset \Lambda_G + j .
\]
Corresponding to a $(G,\delta)$-equidistributed sequence we define for $L \in G \NN$ the set
\[
W_\delta (L) = \bigcup_{j \in (G \ZZ)^d } B(\x_j , \delta) \cap \Lambda_L .
\]
\end{definition}
\begin{theorem} \label{thm:result1}
There is $N = N(d)$ such that for all $\delta \in (0,1/2)$, all $(1,\delta)$-equidistributed sequences, all measurable and bounded $V: \RR^d \to \RR$, all $L \in \NN$, all $b \geq 0$ and all $\phi \in \mathrm{Ran} (\chi_{(-\infty,b]}(H_L))$ we have
\begin{equation}
\label{eq:result1}
\lVert \phi \rVert_{L^2 (W_\delta (L))}^2
\geq C_\sfuc \lVert \phi \rVert_{L^2 (\Lambda_L)}^2
\end{equation}
where
\[
 C_\sfuc = C_{\sfuc} (d, \delta ,  b  , \lVert V \rVert_\infty ) := \delta^{N \bigl(1 + \lVert V \rVert_\infty^{2/3} + \sqrt{b} \bigr)}.
\]
\end{theorem}
For $t,L > 0$ and a measurable and bounded $V : \RR^d \to \RR$ we define the Schr\"odinger operator $H_{t,L} = -t \Delta_L + V_L$ on $L^2 (\Lambda_L)$.
By scaling we obtain the following corollary\ifthenelse{\boolean{journal}}{}{, see Appendix~\ref{sec:proof:cor}}.%
\begin{corollary} \label{cor:result1}
Let $N = N(d)$ be the constant from Theorem~\ref{thm:result1}. Then, for all $G,t > 0$, all $\delta \in (0,G/2)$, all $(G,\delta)$-equidistributed sequences, all measurable and bounded $V: \RR^d \to \RR$, all $L \in G\NN$, all $b \geq 0$ and all $\phi \in \mathrm{Ran} (\chi_{(-\infty,b]}(H_{t,L}))$ we have
 \begin{equation*}
\lVert \phi \rVert_{L^2 (W_\delta (L))}^2
\geq C_{\sfuc}^{G,t} \lVert \phi \rVert_{L^2 (\Lambda_L)}^2
\end{equation*}
where
\begin{equation*}
C_{\sfuc}^{G,t} = C_{\sfuc}^{G,t} (d, \delta ,  b  , \lVert V \rVert_\infty )
:=  \left(\frac{\delta}{G} \right)^{N \bigl(1 + G^{4/3} \lVert V \rVert_\infty^{2/3}/t^{2/3} + G \sqrt{ b / t} \bigr)} .
\end{equation*}
\end{corollary}
Note that the set $W_\delta(L)$ depends on $G$ and the choice of the $(G,\delta)$-equidistributed sequence.
In particular, there is a constant $M = M (d,G,t) \geq 1$ such that
\begin{equation} \label{eq:sfuc}
C_{\sfuc}^{G,t} \geq \delta^{M \bigl(1 + \lVert V \rVert_\infty^{2/3} + \sqrt{\lvert b \rvert} \bigr)}.
\end{equation}
Note that Theorem~\ref{thm:result1} and Corollary~\ref{cor:result1} also hold for $b < 0$, since
\[
 \mathrm{Ran} (\chi_{(-\infty,b]}(H)) \subset \mathrm{Ran} (\chi_{(-\infty,0]}(H))
\]
for any self-adjoint operator $H$.

\begin{remark}[Previous results]
\label{r:previous}
If $L=G$ the result is closely related to doubling estimates and bounds on the vanishing order,
cf.~\cite{LebeauR-95,Kukavica-98,JerisonL-99,Bakri-13}. 
These results, however, do not study the dependence of the bound on geometric data,
e.g.\ the diameter of the domain or manifold.
In the context of random Schr\"odinger  operators
results like \eqref{eq:result1}
have been proven before under additional assumptions and using other methods:
For $V\equiv 0$ and energies close to the minimum of the spectrum in \cite{Kirsch-96} and \cite{BourgainK-05}; 
near spectral edges of periodic Schr\"odinger operators in \cite{KirschSS-98a};
and for periodic geometries $W_\delta$ and potentials in \cite{CombesHK-03}.
 More recently and using similar methods as we do, bounds like \eqref{eq:result1} have been established  for individual eigenfunctions in \cite{RojasMolinaV-13}.
 This has then been extended in \cite{Klein-13} to linear combinations of eigenfunctions of closeby eigenvalues.
 For more references and a broader discussion of the history see e.g.\ \cite{RojasMolinaV-13}, \cite{Klein-13}, or \cite{TaeuferTV-16}.
\end{remark}
As an application to spectral theory we have the following corollary. A proof is given at the end of Section~\ref{section:proof_main_result}.
\begin{corollary} \label{cor:eigenvalue}
Let $b , \alpha , G > 0$, $\delta \in (0,G/2)$, $L \in G \NN$ and $A_L,B_L : \Lambda_L \to \RR$ be measurable, bounded and assume that
\[
B_L \geq  \alpha \chi_{W_\delta(L)}
\]
for a $(G,\delta)$-equidistributed sequence.
Denote the eigenvalues of a self-adjoint operator $H$ with discrete spectrum by $\lambda_i(H)$, enumerated increasingly and counting multiplicities.
Then for all $i \in \NN$ with $\lambda_i(- \Delta + A_L + B_L) \leq b$, we have
\[
\lambda_i(-\Delta_L + A_L + B_L)
\geq
\lambda_i(- \Delta_L + A_L)
+
\alpha C_\sfuc^{G,1}(d,\delta,b,\lVert A_L + B_L \rVert_\infty ) .
\]
\end{corollary}
%
%%%%%%%%%%%%
%%%
%%% Application to Schr\"odinger  operators
%%%
%%%%%%%%%%%%
\subsection{Application to random breather Schr\"odinger operators}
\label{ss:standard}
An important application of our result is in the spectral theory of random Schr\"odinger  operators.
The above scale-free unique continuation estimate is the key for proving the Wegner estimate formulated below,
which is a bound on the expected number of eigenvalues in a short energy interval of a finite box restriction of our random Hamiltonian.
Together with a so-called initial scale estimate, Wegner estimates facilitate a proof of
Anderson localization via multi-scale analysis.
For more background on multi-scale analysis \& localization and on Wegner estimates consult
e.g.\ the monographs \cite{Stollmann-01} and \cite{Veselic-08}, respectively.
\par
The main point is that the potentials we are dealing with here exhibit a \emph{non-linear dependence} on the random parameters $\omega_j$.
Due to this challenge, previously established versions of  \eqref{eq:result1}, as discussed in Remark \ref{r:previous}, are not sufficiently precise to be applied to such models.
We emphasize that our scale-free unique continuation principle and Wegner estimate are valid for all bounded energy intervals, not only near the bottom of the spectrum.
\par
Let us introduce a simple, but paradigmatic example of the models we are considering. (The general case will be studied in the next paragraph.)
\par
Let $\mathcal{D} $ be a countable set to be specified later.
For $0 \leq \omega_{-} < \omega_{+} < 1$ we define the probability space $(\Omega , \mathcal{A} , \PP)$ with
\[
 \Omega = \bigtimes_{j \in \mathcal{D}} \RR , \quad
 \mathcal{A} = \bigotimes_{j \in \mathcal{D}} \mathcal{B} (\RR)
 \quad \text{and} \quad
 \PP = \bigotimes_{j \in \mathcal{D}} \mu ,
\]
where $\mathcal{B}(\RR)$ is the Borel $\sigma$-algebra and $\mu$ is a probability measure with $\mathrm{supp}\ \mu \subset [\omega_{-}, \omega_{+}]$ and a bounded density $\nu_{\mu}$.
Hence, the projections $\omega \mapsto \omega_k$ give rise to a sequence of independent and identically distributed random variables $\omega_j$, $j \in \mathcal{D}$.
We denote by $\EE$ the expectation with respect to the measure $\PP$.
\par
The standard random breather model is defined as
\begin{equation}\label{eq:standardRBP}
H_\omega = -\Delta + V_\omega(x), \qquad \text{ with }     V_\omega(x) = \sum_{j \in \ZZ^d} \chi_{B_{\omega_j}}(x-j)
\end{equation}
and the restriction of $H_\omega$ to the box $\Lambda_L$ by $H_{\omega,L} $. Here obviously $\mathcal{D} =\ZZ^d$.
Denote by $\chi_{[E- \epsilon, E + \epsilon]}$ the spectral projector of  $H_{\omega,L} $.
We formulate now a version of our general Theorem \ref{thm:wegner}
applied to the standard random breather model.
\begin{theorem}[Wegner estimate for the standard random breather model]\label{thm:Wegner}
Assume that  $[\omega_{-}, \omega_{+}] \subset [0,1/4]$, fix $E_0 \in \RR$,
and set $ \epsilon_{\max} = \frac{1}{4}\cdot  8^{-N(2+{\lvert E_0+1 \rvert}^{1/2})}$,
where $N$ is the constant from Theorem~\ref{thm:result1}.
Then there is  $C=C(d,E_0) \in (0,\infty)$ such that for all $\epsilon \in (0,\epsilon_{\max}]$ and $E \geq 0$ with
$[E-\epsilon, E+\epsilon] \subset (- \infty , E_0]$,
we have
\begin{equation*}
\EE \left[ \mathrm{Tr} \left[ \chi_{[E- \epsilon, E + \epsilon]}(H_{\omega,L}) \right] \right]
\leq
C
\lVert \nu \rVert_\infty
\epsilon^{[N(2+{\lvert E_0 + 1 \rvert}^{1/2})]^{-1}}
\left\lvert\ln \epsilon \right\rvert^d L^d.
\end{equation*}
\end{theorem}
Theorem \ref{thm:Wegner} implies local H\"older continuity  of the integrated density of states (IDS) and
is sufficient for the multi-scale analysis proof of spectral localization, see the next paragraph.
 \begin{remark}[Previous results on the random breather model]
 The paper \cite{CombesHM-96} introduced random breather potentials,
while a Wegner estimate was proven in \cite{CombesHN-01},
 however excluding any bounded and any continuous single site potential, cf.~Appendix \ref{app:assumption_single_site}.
Lifshitz tails for random breather Schr{\"o}dinger operators were proven in \cite{KirschV-10}.
All of the papers mentioned so far approached the breather model using techniques which have been developed for the alloy type model.
Consequently, at some stage the non-linear dependence on the random variables was linearised, giving rise to certain differentiability conditions.
As a result, characteristic functions of cubes or balls which would be the most basic example one can think of were excluded as single-site potentials.
Only \cite{Veselic-07a} considers a simple non-differentiable example, namely the standard random breather potential in one dimension, and proves a Lifshitz tail estimate.
 \end{remark}
\subsection{More general non-linear models and localization}
We formulate now a Wegner estimate for a general class of models, which includes the standard random breather potential,
considered in the last paragraph as a special case. We state also an initial scale estimate which implies localization.
\par
Here, in the general setting, we assume that $\mathcal{D} \subset \RR^d$ is  a Delone set, i.e.\ there are $0< G_1<G_2$ such that for any $x \in \RR^d$, we have $\lvert \{ \mathcal{D} \cap (\Lambda_{G_1} + x) \} \rvert \leq 1$ and $\lvert \{ \mathcal{D} \cap (\Lambda_{G_2} + x) \} \rvert \geq 1$. Here, $| \cdot |$ stands for the cardinality. In other words, Delone sets are relatively dense and uniformly discrete subsets of $\mathbb{R}^d$. 
For more background about Delone sets, see, for example, the contributions in  \cite{KellendonkLS-15}.
The reader unacquainted with the concept of a Delone set can always think of $\mathcal{D} = \ZZ^d$.
\par
Furthermore, let $\{ u_t : t \in [0,1] \} \subset L_0^\infty(\RR^d)$ be functions such that there are $G_u \in \NN$, $u_{\max} \geq 0$, $\au, \bu> 0$ and $\ao, \bo \geq 0$ with
\begin{equation} \label{eq:condition_u}
 \begin{cases}
    \displaystyle{ \forall  t \in [0,1]:\ \supp u_t \subset \Lambda_{G_u}} ,\\
    \displaystyle{ \forall  t \in [0,1]:\ \lVert u_t \rVert_\infty \leq u_{\max}},\\
    \displaystyle{ \forall  t \in [\omega_{-}, \omega_{+}],\ \delta \leq 1 - \omega_{+}:\
	    \exists x_0 \in \Lambda_{G_u}:\
		u_{t + \delta} - u_t \geq \au  \delta^{\ao} \chi_{B(x_0, \bu  \delta^{\bo}) }} .
 \end{cases}
\end{equation}
We define the family of Schr\"odinger operators $H_\omega$, $\omega \in \Omega$, on $L^2(\RR^d)$ given by
\[
 H_\omega := - \Delta + V_\omega \quad \text{where} \quad V_\omega(x) = \sum_{j \in \mathcal{D}} u_{\omega_j} (x-j) .
\]
Note that for all $\omega \in [0,1]^{\mathcal{D}}$ we have $\lVert V_\omega \rVert_\infty \leq K_u := u_{\max}  \lceil G_u/G_1 \rceil^d$, c.f.\ Lemma \ref{lem:wegner}.
Assumption \eqref{eq:condition_u} includes many prominent models of random Schr\"o\-ding\-er operators - linear and non-linear. We give some examples.
 \begin{description}[\setleftmargin{0pt}\setlabelstyle{\bfseries}]
  \item[Standard random breather model:]Let $\mu$ be the uniform distribution on $[0, 1/4]$ and let $u_t(x) = \chi_{B(0, t)}$, $j \in \ZZ^d$. Then $V_\omega = \sum_{j \in \ZZ^d} \chi_{B(j,\omega_j)}$ is the characteristic function of a disjoint union of balls with random radii. 
  Such models were introduced  in \S \ref{ss:standard}.
  \item[General random breather models]
  Let $0 \leq u \in L_0^\infty(\RR^d)$ and define $u_t(x) := u(x/t)$ for $t > 0$ and $u_{j,0} :\equiv 0$, $j \in \ZZ^d$, and assume that the family $\{ u_t : t \in [0,1] \}$ satisfies \eqref{eq:condition_u}. 
Natural examples are discussed in Appendix~\ref{app:assumption_single_site}.
 Then $V_\omega(x) = \sum_{j \in \ZZ^d} u_{\omega_j}(x-j)$ is a sum of random dilations of a single-site potential $u$ at each lattice site $j \in \ZZ^d$.
  \item[Alloy type model]
  Let $0 \leq u \in L_0^\infty(\RR^d)$, $u \geq \alpha > 0$ on some open set and let $u_t(x) := t u(x)$. Then $V_\omega(x) = \sum_{j \in \ZZ^d} \omega_j u(x-j)$ is a sum of copies of $u$ at all lattice sites $j \in \ZZ^d$, multiplied with $\omega_j$.
  \item[Delone-alloy type model]
  Let $\mathcal{D} \subset \RR^d$ be a Delone set, $0 \leq u \in L_0^\infty(\RR^d)$, $u \geq \alpha > 0$ on some nonempty open set and let $u_t(x) := t u(x)$. Then $V_\omega(x) = \sum_{j \in \mathcal{D}} \omega_j u(x-j)$ is a sum of copies of $u$ at all lattice sites $j \in \mathcal{D}$, multiplied with $\omega_j$.
  See \cite{GerminetMRM-15} and the references therein for background on such models.
 \end{description}
 For $L > 0$ we denote by $H_{\omega , L}$ the restriction of $H_\omega$ to $L^2 (\Lambda_L)$ with Dirichlet boundary conditions. Following the methods developed in \cite{HundertmarkKNSV-06}, we obtain a Wegner estimate under our general assumption \eqref{eq:condition_u}.
\begin{theorem}[Wegner estimate]\label{thm:wegner}
For all $b \in \RR$ there are constants $C,\kappa,\epsilon_{\max} > 0$, depending only on $d$, $b$, $K_u$, $G_u$, $G_2$, $\au$, $\ao$, $\bu$, $\bo$, $\omega_{+}$ and $\lVert \nu_\mu \rVert_\infty$, such that for all $L \in (G_2 +G_u) \NN$, all $E \in \RR$ and all $\epsilon \leq \epsilon_{\max}$ with $[E - \epsilon, E + \epsilon] \subset (- \infty,b]$ we have
 \begin{equation}
  \EE \left[ \mathrm{Tr} \left[ \chi_{( - \infty, b]} (H_{\omega, L}) \right] \right] \leq C
  \epsilon^{1/\kappa}
  \left\lvert\ln \epsilon \right\rvert^d
   L^d.
 \end{equation}
\end{theorem}
\begin{theorem}[Initial scale estimate] \label{thm:initial}
Let $\kappa$ be as in Theorem~\ref{thm:wegner} for $b = d\pi^2 + K_u$. Assume that there are $t_0,C > 0$ such that
\[
 0 \in \supp \mu
 \quad \text{and} \quad
 \forall t \in [0,t_0] \colon \quad \mu ([0,t]) \leq C t^{d\kappa} .
\]
Then there is $L_0 = L_0 (t_0 , \delta_{\mathrm{max}} , \kappa,  G_u , G_1) \geq 1$ such that for all $L \in (G_2 + G_u) \NN$, $L \geq L_0$ we have
\[
 \PP \left(\left\{ \omega \in \Omega \colon \lambda_1 (H_{\omega , L}) - \lambda_1 (H_{0 , L}) \geq \frac{1} {L^{3/2}} \right\}  \right) \geq 1 - \frac{C}{L^{d/2}} ,
\]
where $H_{0,L}$ is obtained from $H_{\omega , L}$ by setting $\omega_j$ to zero for all $j \in \mathcal{D}$.
\end{theorem}
\begin{remark}[Discussion on initial scale estimate]
Theorem~\ref{thm:initial} may serve as an initial scale estimate for a proof of localization via multi-scale analysis. More precisely, by using the Combes-Thomas estimate, an initial scale estimate in some neighbourhood of $a:= \inf \sigma (H_0)$ follows. Note that the exponents $3/2$ and $d/2$ in Theorem~\ref{thm:initial} can be modified to some extent by adapting the proof and the assumption on the measure $\mu$. Localization in a neighbourhood of $a$ follows via multi-scale analysis, e.g., \`a la \cite{Stollmann-01}. The question whether $\sigma (H_\omega) \cap I_a \not = \emptyset$ for almost all $\omega \in \Omega$ has to be settled. This is, however, satisfied for all examples mentioned above.
In the special case of the standard random breather model one can get rid of the assumption on $\mu$ by proving and using the Lifshitz tail behaviour of the integrated density of states, cf. \cite{Veselic-07a} for the one-dimensional case, and the forthcoming  \cite{SchumacherV} for the multidimensional one.
\end{remark}
\subsection{Application to control theory}
We consider the controlled heat equation
\begin{equation} \label{eq:parabolic}
\begin{cases}
\partial_t u - \Delta u + Vu = f\chi_{\omega}, & u \in L^2([0,T] \times \Omega),  \\
u = 0, & \text{on}\ (0,T) \times \partial \Omega ,  \\
u(0,\cdot) =u_0, & u_0\in L^2(\Omega),
 \end{cases}
 \end{equation}
where $\omega$ is an open subset of the connected $\Omega \subset \RR^d$, $T > 0$ and $V\in L^{\infty}(\Omega)$.
In \eqref{eq:parabolic} $u$ is the state and $f$ is the control function which acts on the system through the control set $\omega$.
\begin{definition}
For initial data $u_0\in L^2(\Omega)$ and time $T > 0$, the set of reachable states $R(T,u_0)$ is
\[
R(T,u_0) = \left\{ u(T,\cdot)\colon \text{there exists}\ f\in L^2([0,T]\times \omega ) \ \text{such that} \ u \ \text{is solution of \eqref{eq:parabolic}} \right\}.
\]
The system \eqref{eq:parabolic} is called null controllable at time $T$ if $0\in R(T;u_0)$ for all $u_0\in L^2(\Omega)$.
The controllability cost $\mathcal{C}(T,u_0)$ at time $T$ for the initial state $u_0$ is
\[
\mathcal{C}(T,u_0) = \inf \left\{ \lVert f \rVert_{L^2([0,T]\times \omega )}\colon
u \
\text{is solution of \eqref{eq:parabolic} and }
u(T,\cdot)=0  \right\}.
\]
\end{definition}
Since the system is linear, null controllability implies that the range of the semigroup generated by the heat equation is reachable too.
It is well known that null controllability holds for any time $T>0$, connected $\Omega$ and any nonempty and open set $\omega \subset \Omega$ on which the control acts, see \cite{FursikovI-96}.

It is also known, see for instance \cite[Theorem 11.2.1]{WeissT-09}, that null controllability of the system \eqref{eq:parabolic} at time $T$ is equivalent to final state observability on the set $\omega$ at time $T$ of the following system:
\begin{equation} \label{eq:control_problem}
\begin{cases}
\partial_t u - \Delta u + Vu = 0, & u \in L^2([0,T] \times \Omega), \\
u = 0, & \text{on} \ (0,T) \times \partial \Omega , \\
u(0,\cdot) = u_0, & u_0\in L^2(\Omega).
\end{cases}
\end{equation}
\begin{definition}
The system \eqref{eq:control_problem} is called final state observable on the set $\omega$ at time $T$ if there exists $\kappa_T = \kappa_T(\Omega, \omega, V)$ such that for every initial state $u_0 \in L^2 (\Omega)$ the solution $u \in L^2([0,T] \times \Omega)$ of \eqref{eq:control_problem} satisfies
\begin{equation}
\label{eq:obscost}
 \lVert u(T,\cdot) \rVert_{\Lambda_L}^2
 \leq \kappa_T  \lVert u \rVert_{L^2([0,T]\times \omega )}^2.
 \end{equation}
\end{definition}
Moreover, the controllability cost $\mathcal{C}(T,u_0)$ of \eqref{eq:parabolic} coincides with the infimum over all observability costs $\sqrt{\kappa_T}$ in \eqref{eq:obscost} times $\lVert u_0 \rVert_{\Lambda_L}$ (see, for example, the proof of \cite[Theorem 11.2.1]{WeissT-09}.
\par
The problem of obtaining explicit bounds on $\mathcal{C}(T,u_0)$ received much consideration in the literature (see, for example, \cite{Guichal-85,FernandezZ-00,Phung-04,TenenbaumT-07,Miller-06,Miller-04,Miller-10,ErvedozaZ-11,Lissy-12}), especially the case of small time, i.e.\ when $T$ goes to zero. 
The dependencies of the controllability cost on $T$ and $\lVert V \rVert_\infty$ are today well understood, see, for example \cite{Zuazua-07}. 
However, the dependence on the geometry of the control set is less clear: in the known estimates the geometry enters only in terms of the distance to the boundary or in terms of the geometrical optics condition.
To find an optimal control set is a very difficult problem, see for instance the recent articles \cite{PrivatTZ-15a, PrivatTZ-15b}.
\par
We are interested in the situation $\Omega = \Lambda_L \subset \RR^d$ and $\omega = W_\delta(L)$ for a $(G,\delta)$-equidistributed sequence with $L \in G\NN$, $G > 0$ and $\delta < G/2$.
In this specific setting we will give an estimate on the controllability cost.
The novelty of our result is that the observability cost is independent
of the scale $L$ and the specific choice of the $(G, \delta)$-equidistributed sequence.
Moreover, the dependencies on $\lVert V \rVert_\infty$ and on the size of the control set via $\delta$ are known explicitly.
As far as we are aware, this is the first time that such a scale-free estimate is obtained.
\par
 By the equivalence between null-controllability and final state observability, it is sufficient to construct an estimate of the form  \eqref{eq:obscost}.
In order to find such an estimate, we will combine Corollary \ref{cor:result1} with results from \cite{Miller-10} to obtain the following theorem.
\begin{theorem}
\label{thm:contcost}
For every $G > 0$, $\delta \in (0, G/2)$ and $K_V \geq 0$ there is $T' = T'(G, \delta, K_V) > 0$ such that
for all $T \in (0,T']$, all $(G,\delta)$-equidistributed sequences, all measurable and bounded $V: \RR^d \to \RR^d$ with $\lVert V \rVert_\infty \leq K_V$ and all $L \in G \NN$, the system
\begin{equation*} \label{eq:control_problem_concrete}
\begin{cases}
\partial_t u - \Delta_L u + V_L u = 0, & u \in L^2([0,T] \times \Lambda_L), \\
u = 0, & \text{on} \ (0,T) \times \partial \Lambda_L , \\
u(0,\cdot) = u_0, & u_0\in L^2(\Lambda_L).
\end{cases}
\end{equation*}
is final state observable on the set $W_\delta(L)$ with cost $\kappa_T$ satisfying
\[
\kappa_T \leq 4 a_0 b_0 \euler^{2 c_{\ast} / T} ,
\]
where
$  a_0 = (\delta / G)^{- N ( 1 + G^{4/3} \lVert V \rVert_\infty^{2/3})}$,
 $b_0 = \euler^{2 \lVert V \rVert_\infty}$,  
 $c_{\ast} \leq \ln (G/\delta)^2  \left( NG + 4 / \ln 2 \right)^2$ {and}
 $N = N (d)$ is the constant from Theorem \ref{thm:result1}.
\end{theorem}
\begin{remark}
The same result holds also in the case of the controlled heat equation with periodic or Neumann boundary conditions with obvious modifications. 	
\end{remark}
\begin{remark}
Null controllability of the heat equation implies a stronger type of controllability, so-called approximate controllability. Following \cite{FernandezZ-00}, one can find an estimate for the cost of approximate controllability from the proof of Theorem \ref{thm:contcost}. We will not pursue it in this paper.
\end{remark}
%
%%%%%%%%%%%%
%%%
%%% Proof of UCP
%%%
%%%%%%%%%%%%
%
\section{Proof of scale-free unique continuation principle}\label{s:proof-sfuc}
%%%%%%%%%%%%
%%%
%%% Carleman inequalities
%%%
%%%%%%%%%%%%
\subsection{Carleman inequalities}
We denote by $\RR^{d+1}_+ := \{ x \in \RR^{d+1} \colon x_{d+1} \geq 0 \}$ the $d+1$-dimensional half-space and by $B_r^+ := \{x \in \RR_+^{d+1} \colon \lvert x \rvert < r\}$ the $d+1$-dimensional half-ball. For $x \in \RR^{d+1}$ we denote by $x'$ the projection on the first $d$ coordinates, i.e.\ for $x = (x_1 , \ldots , x_{d+1}) \in \RR^{d+1}$ we use the notation $x' = (x_1 , \ldots , x_d) \in \RR^d$. By $\lvert x \rvert$ and $\lvert x' \rvert$ we denote the Euclidean norms and by $\Delta$ the Laplacian on $\RR^{d+1}$.
For functions $f \in C^\infty (\RR^{d+1}_+)$ we use the notation $f_0 = f\rvert_{x_{d+1} = 0}$.
\par
In the appendix of \cite{LebeauR-95} Lebeau and Robbiano state a Carleman estimate for complex-valued functions with support in $B_r^+$ by using a real-valued weight function $\psi \in C^\infty (\RR^{d+1})$ satisfying the two conditions
\begin{equation} \label{eq:7}
 \forall x \in B_r^+ \colon \quad (\partial_{d+1} \psi ) (x) \not = 0,
\end{equation}
and for all $\xi \in \RR^{d+1}$ and $x \in B_r^+$ there holds
\begin{equation} \label{eq:8}
\left.
\begin{array}{l}
 2 \langle \xi , \nabla \psi \rangle = 0 \\[1ex]
 \lvert \xi \rvert^2 = \lvert \nabla \psi \rvert^2
\end{array}
\right\}
\quad \Rightarrow \quad
\sum_{j,k=1}^{d+1} (\partial_{jk} \psi) \bigl(\xi_j \xi_k + (\partial_j \psi) (\partial_k \psi) \bigr) > 0 .
\end{equation}
As proposed in \cite{JerisonL-99} we choose $r < 2 - \sqrt{2}$ and the special weight function $\psi : \RR^{d+1} \to \RR$,
\begin{equation} \label{eq:weight}
\psi (x) = -x_{d+1} + \frac{x_{d+1}^2}{2} - \frac{\lvert x' \rvert^2}{4} .
\end{equation}
Note that $\psi (x) \leq 0$ for all $x \in B_2^+$.
This function $\psi$ indeed satisfies the assumptions \eqref{eq:7} and \eqref{eq:8}. Condition \eqref{eq:7} is trivial for $r < 1$. In order to show the implication \eqref{eq:8} we show
\begin{equation} \label{eq:8b}
 \lvert \xi \rvert^2 = \lvert \nabla \psi \rvert^2  \ \Rightarrow \
\sum_{j,k=1}^{d+1} \partial_{jk} \psi (\xi_j \xi_k + \partial_j \psi \partial_k \psi) > 0 .
\end{equation}
We use the hypothesis of \eqref{eq:8b} and calculate
\begin{align*}
\sum_{j,k=1}^{d+1} \partial_{jk} \psi (\xi_j \xi_k + \partial_j \psi \partial_k \psi)
&= -\frac{1}{2} \sum_{i=1}^d \xi_i^2 + \xi_{d+1}^2 - \frac{1}{8} \lvert x' \rvert^2 + (x_{d+1} - 1)^2  \\
&= \frac{3}{2} \xi_{d+1}^2 - \frac{1}{4} \lvert x' \rvert^2 + \frac{1}{2} (x_{d+1}-1)^2 .
\end{align*}
Since $\lvert x' \rvert^2 \leq r^2$ and $(x_{d+1}-1)^2 \geq (1 - r)^2$, assumption \eqref{eq:8b} is satisfied if $r < 2 - \sqrt{2}$. Now let
\begin{multline*}
 C_{\mathrm{c} , 0}^\infty (B_r^+)
 =
 \left\{ g : \RR^{d+1}_+ \to \CC \colon g \equiv 0 \ \text{on} \ \{x_{d+1} = 0\}, \right. \\
 \left. \exists \psi \in C^\infty (\RR^{d+1}) \ \text{with}\ \supp \psi \subset \{x \in \RR^{d+1} \colon \lvert x \rvert < r\} \ \text{and} \ g \equiv \psi \ \text{on} \ \RR_+^{d+1} \right\} .
\end{multline*}
Hence, as a corollary of Proposition~1 in the appendix of \cite{LebeauR-95} we have
\begin{proposition} \label{prop:carleman}
Let $\psi \in C^\infty (\RR^{d+1} ; \RR)$ be as in Eq.~\eqref{eq:weight} and $\rho \in (0,2-\sqrt{2})$. Then there are constants $\beta_0, C_1 \geq 1$ such that for all $\beta \geq \beta_0$, and all
$g \in C_{\mathrm c , 0}^\infty (B_\rho^+)$ we have
\begin{equation*}
 \int_{\RR^{d+1}} \euler^{2\beta \psi} \left( \beta  \lvert \nabla g \rvert^2 + \beta^3 \lvert g \rvert^2 \right) \leq
 C_1 \left(
 \int_{\RR^{d+1}} \euler^{2 \beta \psi} \lvert \Delta g \rvert^2  +  \beta \int_{\RR^d} \euler^{2 \beta \psi_0} \lvert (\partial_{d+1} g)_0 \rvert^2
 \right).
\end{equation*}
\end{proposition}
We will also need the following Carleman estimate.
\begin{proposition} \label{prop:carleman2}
Let $\rho > 0$. Then there are constants $\alpha_0,C_2 \geq 1$ depending only on the dimension and a function $w = \:\RR^d \to \RR$ satisfying
\[
 \forall x \in B(\rho) \colon \frac{\lvert x \rvert}{\rho \mathrm{e}} \leq w(x) \leq  \frac{\lvert x \rvert}{\rho} ,
\]
such that for all $\alpha \geq \alpha_0$, and all $u \in W^{2,2} (\RR^d)$ with support in $B(\rho) \setminus \{0\}$ we have
\[
 \int_{\RR^d} \left( \alpha {\rho^2} w^{1-2\alpha} \lvert \nabla u \rvert^2 + \alpha^3 w^{-1-2\alpha} \lvert u \rvert^2  \right) \drm x \leq C_2 \rho^4 \int_{\RR^d} w^{2-2\alpha} \left\lvert \Delta u \right\rvert^2 \drm x .
\]
\end{proposition}
Proposition~\ref{prop:carleman2} is a special case of the result obtained in \cite{NakicRT-15} where general second order elliptic partial differential operators with Lipschitz continuous coefficients are considered.
The estimate has been previously obtained; (1) in \cite{BourgainK-05}, but there without the gradient term on the left hand side; (2) in \cite{EscauriazaV-03}, but there without a quantitative statement of the admissible functions $u$.
These weaker versions are not sufficient for our purposes.
In Appendix~\ref{app:carleman} we sketch for reader acquainted with the proof of \cite{BourgainK-05} the difference between the two results.
%
%%%%%%%%%%%%
%%%
%%% Extension to larger boxes
%%%
%%%%%%%%%%%%
%
\subsection{Extension to larger boxes}\label{sec:extension}
For each measurable and bounded $V:\RR^d \to \RR$ and each $L \in \NN$ we denote the eigenvalues of the corresponding operator $H_L$ by $E_k$, $k \in \NN$, enumerated in increasing order and counting multiplicities, and fix a corresponding sequence $\phi_k$, $k \in \NN$, of normalized eigenfunctions. Note that we suppress the dependence of $E_k$ and $\phi_k$ on $V$ and $L$.
\par
Given $V$ and $L$ we define an extension of the potential $V_L$ and the eigenfunctions $\phi_k$ to the set $\Lambda_{\R L}$ for some $\R \in \NN_{\mathrm{odd}} = \{1,3,5,\ldots\}$ to be chosen later on. The extension will depend on the type of boundary conditions we are considering for the Laplace operator.
\begin{description}[\setleftmargin{0pt}]
\item[Extension for periodic boundary conditions:]
We extend the potential $V_L$ as well as the function $\phi_k$,
defined on the box $\Lambda_L$, periodically to  $\tilde V, \tilde \psi \colon \RR^d\to\RR$ and then restrict them to $\Lambda_{RL}$.
By the very definition of the operator domain of $\Delta_{\Lambda_L}$
with periodic boundary conditions
the extension $\tilde \psi$ is locally in the Sobolev space $W^{2,2}(\RR^d)$.
\item[Extension for Dirichlet and Neumann boundary conditions:]
The potential $V_L$ will be ex\-ten\-ded by symmetric reflections with respect to the hypersurfaces forming the boundaries of $\Lambda_L$. In the first step we extend $V_L : \Lambda_L \to \RR$ to the set
 $H_L = \{x \in \Lambda_{3L} \colon  x_i \in (-L/2 , L/2), \ i \in \{2,\ldots , d\}\}$ by
\[
V_L (x) = \begin{cases}
V_L (x) & \text{if $x \in \Lambda_{L}$} , \\
0 & \text{if $x_1 \in \{-L/2 , L/2\}$}, \\
V_L (L - x_1 , x_2 , \ldots , x_d) & \text{if $x_1 > L/2$} , \\
V_L (-L - x_1 , x_2 , \ldots , x_d) & \text{if $x_1 < - L/2$}.
\end{cases}
\]
Now we iteratively extend $V_L$ in the remaining $d-1$ directions using the same procedure and obtain a function $V_L : \Lambda_{3L} \to \RR$. Iterating this procedure we obtain a function $V_L : \Lambda_{\R L} \to \RR$.
The extensions of the eigenfunctions will depend on the boundary conditions.
In the case of Dirichlet boundary conditions, we extend an eigenfunction similarly to the potential by antisymmetric reflections, while in the case of Neumann boundary conditions, we extend by symmetric reflections.
\end{description}
The extensions of the functions and $V_L$ and $\phi_k$, $k \in \NN$, to the set $\Lambda_{\R L}$ will again be denoted by $V_L$ and $\phi_k$, $k \in \NN$. The reader should be reminded that (the extended) $V_L : \Lambda_{RL} \to \RR$ does in general not coincide with $V_{RL}: \Lambda_{RL} \to \RR$.
Note that for all three boundary conditions, $V_{L}: \Lambda_{RL} \to \RR$ takes values in $[-\lVert V \rVert_\infty, \lVert V \rVert_\infty]$, the extended $\phi_k$ are elements of $W^{2,2} (\Lambda_{\R L})$ with corresponding boundary conditions and they satisfy $\Delta \phi_k = (V_L - E_k) \phi_k$ on $\Lambda_{\R L}$.
Furthermore, the orthogonality relations remain valid.
%
%%%%%%%%%%%%
%%%
%%% Ghost dimension
%%%
%%%%%%%%%%%%
%
\subsection{Ghost dimension}
For a measurable and bounded $V : \RR^d \to \RR$, $L \in \NN$, $b \geq 0$ and $\phi \in \mathrm{Ran} (\chi_{(-\infty,b]}(H_L))$ we have
\[
 \phi = \sum_{\genfrac{}{}{0pt}{2}{k \in \NN}{E_k \leq b}} \alpha_k \phi_k , \quad \text{with} \quad
 \alpha_k = \langle \phi_k , \phi \rangle .
\]
Since $\phi_k$ extend to $\Lambda_{RL}$ as explained in Section~\ref{sec:extension}, the function $\phi$ also extends to $\Lambda_{RL}$.
We set $\omega_k := \sqrt{\lvert E_k \rvert}$ and define the function $F : \Lambda_{\R L} \times \RR \to \CC$ by
\begin{equation*} \label{eq:F}
 F (x , x_{d+1}) = \sum_{\genfrac{}{}{0pt}{2}{k \in \NN}{E_k \leq b}} \alpha_k \phi_k (x) \funs_k( x_{d+1}) ,
\end{equation*}
where $s_k : \RR \to \RR$ is given by
\[
\funs_k(t)=\begin{cases}
	\sinh(\omega_k t)/\omega_k, & E_k>0,\\
	x, & E_k=0,\\
	\sin(\omega_k t)/\omega_k, & E_k<0.
\end{cases}
\]
Note that we suppress the dependence of $\phi$ and $\phi_k$ on $V$, $L$, $b$. Furthermore, the sums are finite since $H_L$ is lower semibounded with purely discrete spectrum. The function $F$ fulfills the handy relations
\begin{equation*} \label{eq:DeltaF}
\Delta F = \sum_{i=1}^{d+1} \partial^2_{i} F  =  V_L F \quad \text{on} \quad  \Lambda_{\R L} \times \RR
\end{equation*}
and
\begin{equation*} \label{eq:F-phi}
\partial_{d+1} F (x,0) = \sum_{\genfrac{}{}{0pt}{2}{k \in \NN}{E_k \leq b}} \alpha_k \phi_k (x) \quad \text{for} \quad x \in \Lambda_{\R L}  .
\end{equation*}
In particular, for all $x \in \Lambda_L$ we have $\partial_{d+1} F (x , 0) = \phi$. This way we recover the original function we are interested in.
\par
Let us also fix the geometry. For $\delta \in (0,1/2)$ we choose
\begin{align*}
\psi_1 &= -\delta^2 / 16, &\psi_2 &= -\delta^2 / 8,  &\psi_3 &= -\delta^2 / 4, \\
%ex
\rin_1     &= \frac{1}{2} - \frac{1}{8}\sqrt{16-\delta^2}, & \rin_2 &=1,  &\rin_3 &= 6 \mathrm{e} \sqrt{d}, \\
\rout_1 &=1 - \frac{1}{4}\sqrt{16-\delta^2}, & \rout_2 &=3 \sqrt{d}, & \rout_3 &= 9 \mathrm{e} \sqrt{d},
\end{align*}
and define for $i \in \{1,2,3\}$ the sets
\begin{align*}
S_i &:= \bigl\{x \in \RR^{d+1} \colon \psi (x) > \psi_i , x_{d+1}  \in [0,1] \bigr\} \subset \RR^{d+1}_+ \\
\intertext{and}
V_i &:= B(\rout_i) \setminus \overline{B(\rin_i)} \subset \RR^{d+1} .
\end{align*}
We also fix $\R$ to be the least odd integer larger than $2\rout_3 + 2$.
For $i \in \{1,2,3\}$ and $x\in\RR^d$ we denote by $S_i (x) = S_i + (x,0)$ and $V_i (x) = V_i + (x,0)$ the translates of the sets $S_i\subset \RR^{d+1}$ and $V_i \subset \RR^{d+1}$.
Moreover, for $L \in \NN $ and an $(1,\delta)$-equidistributed sequence $z_j \in \RR^d$, $j \in \ZZ^d$, we define $Q_L = \ZZ^d \cap \Lambda_L$, $U_i (L) = \cup_{j \in Q_L} S_i (\x_j)$, $X_1 = \Lambda_L \times [-1,1]$ and $\tilde X_{\rout_3} = \Lambda_{L + 2\rout_3} \times [-\rout_3 , \rout_3]$. Note that $W_\delta (L)$ is a disjoint union. In the following lemma we collect some consequences of our geometric setting.
We will first restrict our attention to the case $L\in \NN_{\mathrm{odd}}$ , and consider the case of even integers thereafter.
\begin{lemma} \label{lemma:geometric}
\begin{enumerate}[(i)]
\item For all $\delta \in (0,1/2)$ we have $S_1 \subset S_2 \subset S_3 \subset B_\delta^+ \subset \RR^{d+1}_+$.
\item For all $L \in \NN_{\mathrm{odd}}$ with $L \geq 5$, all $\delta \in (0,1/2)$ and all $(1,\delta)$-equidistributed sequences $z_j$ we have
 $\cup_{j \in Q_{L}} V_2 (\x_j) \supset X_1$.
\item There is a constant $K_d$, depending only on $d$, such that for all $L \in \NN_{\mathrm{odd}}$, all $\delta \in (0,1/2)$, all $(1,\delta)$-equidistributed sequences $z_j$, all measurable and bounded $V : \RR^d \to \RR$, all $b \geq 0$ and all $\phi \in \mathrm{Ran} (\chi_{(-\infty,b]}(H_L))$ we have
\[
\sum_{j \in Q_L} \lVert F \rVert_{H^1 (V_3 (\x_j))}^2 \leq K_d \lVert F \rVert^2_{H^1 (\cup_{j \in Q_L} V_3 (\x_j))}.
\]
\item For all $L \in \NN_{\mathrm{odd}}$, $\delta \in (0,1/2)$ and all $(1,\delta)$-equidistributed sequences $z_j$ we have $\cup_{j \in Q_L} \allowbreak V_3 (\x_j) \allowbreak \subset \allowbreak \tilde X_{\rout_3}$.
\end{enumerate}
\end{lemma}

We note that part (ii) of Lemma \ref{lemma:geometric} will be applied with $L$ replaced by $5L$.
\begin{proof}
Parts (i) and (iv) are obvious.
\par
To show (ii), we first prove that $[-1/2,1/2]^d\times[-1,1]$ can be covered by the sets $V_2(\x_j)$. Let us take $j_1=(-1,0,\ldots,0),j_2=(-2,0,\ldots,0)$, $j_1,j_2\in Q_L$.
\begin{figure}[b]\centering
 \begin{tikzpicture}[scale=1.0]
  \pgfmathsetseed{{\number\pdfrandomseed}}
  \draw[fill=black!10] (4,0.5) rectangle (5,1.0);
  \draw (5.47,0.75) node {$X_1$};
  \draw[pattern= north east lines,pattern color=black!50] (4,-0.25) rectangle (5,0.25);
  \draw (5.75,0) node {$V_2 (\x_{j_1})$};
  \draw[pattern= north west lines, pattern color=black!50] (4,-1) rectangle (5,-0.5);
  \draw (5.75,-0.75) node {$V_2 (\x_{j_2})$};
  \filldraw[black!10] (-2.5,-1) rectangle (2.5,1);
  \draw[very thick] (-2.5,-1) rectangle (2.5,1);
  \draw[-latex, thick] (-3,0)--(3,0) node[below] {$x_1$};
  \draw[-latex, thick] (0,-1.6)--(0,1.6) node[left] {$x_2$};
  \foreach \x in {-1.5,-0.5,0.5,1.5}{
  \draw[dashed, very thick] (\x , -1)--(\x , 1);
  }
  \foreach \x in {0,1,2}{
    \pgfmathsetmacro{\a}{rand*0.5}
    \filldraw (\x + \a,0) circle (1.5pt);
  }
%  north west lines
    \pgfmathsetmacro{\a}{rand*0.5}
    \clip(-5 + \a,-1.3) rectangle (3 ,1.3);
      \draw[pattern=north west lines, pattern color=black!50] (-2 + \a + 1,0) arc (0:360:1) ++(2,0) arc (360:0:3) -- (-2 + \a , 0) -- cycle;
     \filldraw[] (-2 + \a,0) circle (1.5pt);
    \pgfmathsetmacro{\a}{rand*0.5}
    \draw[pattern= north east lines, pattern color=black!50] (-1 + \a + 1,0) arc (0:360:1) ++(2,0) arc (360:0:3) -- (-1 + \a , 0) -- cycle;
     \filldraw[] (-1 + \a,0) circle (1.5pt);
\end{tikzpicture}
\caption{Illustration for (ii) in case $d=1$, $L=5$ and some configuration $\x_j$, $j \in Q_L$. The set $[-1/2 , 1/2] \times [-1,1]$ is covered by $V_2 (\x_{j_1})$ and $V_2 (\x_{j_2})$.}
\label{fig:covering}
\end{figure}
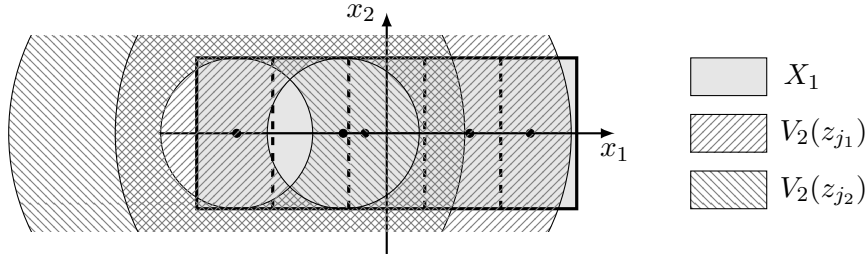
Then
\begin{equation}
\label{eq:cover}
[-1/2,1/2]^d\times[-1,1]\subset V_2(\x_{j_1})\cup V_2(\x_{j_2}) ,
\end{equation}
cf.\ Fig.~\ref{fig:covering}.
Indeed, let $x=(x_1,\ldots,x_{d+1})$ be an arbitrary point from $[-1/2,1/2]^d\times[-1,1]$. Then \eqref{eq:cover} is not satisfied only if
$\lvert (\x_{j_1},0) - x\rvert^2<1$ and $\lvert (\x_{j_2},0) - x\rvert^2>\rout_2^2$.
Since $\x_{j_1}\in ( -3/2+\delta, -1/2 - \delta)\times (-1/2+\delta,1/2 - \delta)^{d-1}$ and $\x_{j_2}\in ( -5/2+\delta, -3/2 - \delta)\times (-1/2+\delta,1/2 - \delta)^{d-1}$, it follows
\begin{gather*}
(-1/2- \delta -x_1)^2+x_{d+1}^2 < 1 \quad\text{and}\quad
(-5/2+ \delta -x_1)^2+(d-1)(1- \delta)^2+x_{d+1}^2 > 9d.
\end{gather*}
Plugging the first relation into the second, we obtain
\[
9d < (d-1)(1 - \delta)^2 + 2 (1- \delta)(3+2x_1) +1 \le (d-1)(1- \delta)^2+8(1- \delta) +1.
\]
But this relation is satisfied only for $d<1$.
Since $L \geq 5$ the same argument applies to cover every elementary cell $([-1/2 , 1/2]+i)\times [-1,1]$, $i \in Q_L$, by two neighboring sets $V_2 (\x_j)$.
\par
Now we turn to the proof of (iii).
Since $\R \geq 2\rout_3 + 2$ the function $F$ is defined on $V_3 (\x_j)$ for all $j \in Q_L$. For all $x \in \cup_{j \in Q_L} V_3 (\x_j)$, the number of indices $j \in Q_L$ such that $V_3 (\x_j) \ni x$ is bounded from above by $(2 \rout_3 + 2)^d$. Hence,
\[
\forall x \in \tilde X_{\rout_3}    \colon \quad \sum_{j \in Q_L} \Eins_{V_3 (\x_j)} (x) \leq (2 \rout_3 + 2)^d \chi_{\cup_{j \in Q_L} V_3 (\x_j)} (x)
\]
and thus
\begin{align*}
\sum_{j \in Q_L} \lVert F \rVert^2_{H^1 (V_3 (\x_j))}
&=
\int_{\tilde X_{\rout_3}} \Bigl( \sum_{j \in Q_L} \Eins_{V_3 (\x_j)} (x) \Bigr) \left( \lvert F(x) \rvert^2 + \lvert \nabla F (x) \rvert^2 \right) \drm x \\
&\leq
(2 \rout_3 + 2)^d \lVert F \rVert^2_{H^1 (\cup_{j \in Q_L} V_3 (\x_j))} .
\end{align*}
Hence we can take $K_d=(2 R_3 + 2)^d$.
\end{proof}
%
%%%%%%%%%%%%
%%%
%%% Interpolation inequalities
%%%
%%%%%%%%%%%%
%
\subsection{Interpolation inequalities}
\begin{proposition} \label{prop:interpolation1}
For all $\delta \in (0,1/2)$, all $(1,\delta)$-equidistributed sequences $z_j$, all measurable and bounded $V: \RR^d \to \RR$, all $L \in \NN_{\mathrm{odd}}$, all $b \geq 0$ and all $\phi \in \mathrm{Ran} (\chi_{(-\infty,b]}(H_L))$
\begin{enumerate}[(a)]
 \item there is $\beta_1 = \beta_1 (d , \lVert V \rVert_\infty) \geq 1$ such that for all $\beta \geq \beta_1$ we have
 \[
 \lVert F \rVert_{H^1 (U_1(L))}^2 \leq \tilde D_1 (\beta) \lVert F \rVert_{H^1 (U_3 (L))}^2 +  \hat D_1 (\beta) \lVert (\partial_{d+1} F)_0 \rVert_{L^2 (W_\delta(L))}^2 ,
\]
where $\beta_1$ is given in Eq.~\eqref{eq:beta1}, and $\tilde D_1(\beta)$ and $\hat D_1(\beta)$ are given in Eq.~\eqref{eq:D_tilde}.
\item we have
\[
  \lVert  F\rVert_{H^1 (U_1 (L))} \leq D_1 \lVert (\partial_{d+1} F)_0 \rVert_{L^2 (W_\delta(L))}^{1/2} \lVert F \rVert_{H^1 (U_3 (L))}^{1/2} ,
\]
where $D_1$ is given in Eq.~\eqref{eq:C2}.
\end{enumerate}
\end{proposition}
\begin{proof}
First we recall that $\Delta F = V_L F$, $\partial_{d+1} F (x',0) = \phi (x')$ and $B_\delta^+ \supset S_3$.
Now we choose a cutoff function $\chi \in C^\infty (\RR^{d+1};[0,1])$ with $\supp \chi \subset \overline{S_3}$, $\chi (x) = 1$ if $x \in S_2$ and
\[
 \max\{\lVert \Delta \chi \rVert_\infty , \lVert \lvert \nabla \chi \rvert \rVert_\infty\} \leq \frac{\tilde\cut_1}{\delta^{4}} =: \cut_1,
\]
where $\tilde\cut_1 = \tilde\cut_1 (d)$ depends only on the dimension\ifthenelse{\boolean{journal}}{. This is due to the fact that the distance of $S_2$ and $\RR^{d+1}_+ \setminus S_3$ is bounded from below by $\delta^2 / 16$}{, see Appendix~\ref{app:constants}}.
Let $\varphi$ be a non-negative function in $C_{\mathrm c}^\infty (\mathbb{R}^d)$ with the properties that $\lVert \varphi \rVert_1 = 1$ and $\supp \varphi \subset B(1)$. For $\epsilon > 0$ we define $\varphi_\epsilon : \mathbb{R}^d \to \mathbb{R}_0^+$ by $\varphi_\epsilon (x) = \epsilon^{-d} \varphi (x/\epsilon)$. The function $\varphi_\epsilon$ belongs to $C_{\mathrm c}^\infty (\mathbb{R}^d)$ and satisfies $\supp \varphi_\epsilon \subset (\epsilon)$. Now we continuously extend the eigenfunctions $\phi_k:\Lambda_{RL} \to \RR$ to the set $\RR^d $ by zero and define for $\epsilon > 0$ the function $F_\epsilon : \RR^d \times \RR$ by
\[
 F_\epsilon (x , x_{d+1})
 =
 \sum_{\genfrac{}{}{0pt}{2}{k \in \NN}{E_k \leq b}} \alpha_k (\varphi_\epsilon \ast \phi_k) (x) \funs_k( x_{d+1}) .
\]
By construction, the function $g = \chi F_\epsilon$ is an element of $C_{\mathrm c , 0}^\infty (B_\delta^+)$. Hence, we can apply
Proposition~\ref{prop:carleman} with $g = \chi F_\epsilon$ and $\rho=1/2$ and obtain for all $\beta \geq \beta_0 \geq 1$
\begin{equation} \label{eq:carleman_apply}
  \int_{S_3} \euler^{2\beta \psi} \left( \beta  \lvert \nabla (\chi F_\epsilon) \rvert^2 + \beta^3 \lvert \chi F_\epsilon\rvert^2 \right)
  \leq
 C_1 \int_{S_3} \euler^{2 \beta \psi} \lvert \Delta (\chi F_\epsilon) \rvert^2  +  \beta C_1 \int_{B (\delta)} \euler^{2 \beta \psi_0} \lvert (\partial_{d+1} (\chi F_\epsilon))_0 \rvert^2
 .
 \end{equation}
Note that $\beta_0$ and $C_1$ only depend on the dimension.
By \cite[Theorem~1.6.1 (iii)]{Ziemer-89} we have $\varphi_\epsilon \ast \phi_k \to  \phi_k$, $\nabla (\varphi_\epsilon \ast \phi_k) \to  \nabla \phi_k$ and $\Delta (\varphi_\epsilon \ast \phi_k) \to \Delta \phi_k$ in $L^2 (S_3)$ as $\epsilon$ tends to zero. Consequently, the same holds for $F_\epsilon$, $\nabla F_\epsilon$ and $\Delta F_\epsilon$ and thus we obtain Ineq.~\eqref{eq:carleman_apply} with $F_\epsilon$ replaced by $F$.
For the first summand on the right hand side we have the upper bound
\begin{align*}
   \int_{S_3} \euler^{2 \beta \psi} \lvert \Delta (\chi F) \rvert^2
   &\leq 3 \int_{S_3} \euler^{2 \beta \psi} \left( 4 \lvert \nabla \chi \rvert^2 \lvert \nabla F \rvert^2 + \lvert \Delta \chi \rvert^2 \lvert F \rvert^2 + \lvert \Delta F\rvert^2 \lvert  \chi \rvert^2 \right) \\
   &\leq 3 \euler^{2 \beta \psi_2} \int_{S_3 \setminus S_2} \left( 4 \cut_1^2  \lvert \nabla F \rvert^2 + \cut_1^2 \lvert F \rvert^2 \right)  + \int_{S_3} 3\euler^{2 \beta \psi} \lvert V_L F \chi \rvert^2  \\
   &\leq 12 \cut_1^2 \euler^{2 \beta \psi_2}  \lVert F \rVert_{H^1 (S_3)}^2 +
        3  \lVert V \rVert_\infty^2 \int_{S_3} \euler^{2 \beta \psi}  \lvert \chi F \rvert^2 .
\end{align*}
The second summand is bounded from above by $\beta C_1 \int_{B(\delta )} \lvert (\partial_{d+1} F)_0 \rvert^2$, since $F = 0$ and $\psi \leq 0$ on $\{x_{d+1} = 0\}$. Hence,
\begin{multline*}
\beta \int_{S_3} \euler^{2\beta \psi}  \lvert \nabla (\chi F) \rvert^2  + (\beta^3 - 3\lVert V \rVert_\infty^2 C_1) \int_{S_3} \euler^{2\beta \psi} \lvert \chi F \rvert^2 \\
\leq
 12 C_1 \cut_1^2 \euler^{2 \beta \psi_2} \lVert F \rVert_{H^1 (S_3)}^2 +
C_1 \beta \lVert (\partial_{d+1} F)_0 \rVert_{L^2 (B (\delta))}^2 .
\end{multline*}
Additionally to $\beta \geq \beta_0$ we choose $\beta \geq (6 \lVert V \rVert_\infty^2 C_1)^{1/3} =: \tilde\beta_0$. This ensures that for all
\begin{equation} \label{eq:beta1}
\beta \geq \beta_1 := \max\{\beta_0 , \tilde\beta_0\}
\end{equation}
we have
\begin{equation*}
\frac{1}{2} \int_{S_3} \euler^{2\beta \psi} \left( \beta \lvert \nabla (\chi F) \rvert^2 + \beta^3\lvert \chi F \rvert^2 \right)
 \leq  12 C_1 \cut_1^2 \euler^{2 \beta \psi_2} \lVert F \rVert_{H^1 (S_3)}^2 +
 C_1 \beta \lVert (\partial_{d+1} F)_0 \rVert_{L^2 (B (\delta))}^2
.
\end{equation*}
Since $\beta \geq 1$, $S_3 \supset S_1$, $\chi = 1$ and $\euler^{2 \beta \psi} \geq \euler^{2 \beta \psi_1}$ on $S_1$, we obtain
\[
\euler^{2\beta \psi_1} \lVert F \rVert_{H^1 (S_1)}^2 \leq
24 C_1 \cut_1^2 \euler^{2 \beta \psi_2}  \lVert F \rVert_{H^1 (S_3)}^2 + 2 C_1 \lVert (\partial_{d+1} F )_0 \rVert_{L^2 (B (\delta))}^2 .
\]
We apply this inequality for translates $S_i (\x_j)$ and obtain by summing over $j \in Q_L = \ZZ^d \cap \Lambda_L$
\begin{equation*}
 \euler^{2\beta \psi_1} \sum_{j \in Q_L} \lVert F \rVert_{H^1 (S_1 (\x_j))}^2  \leq 24 C_1 \cut_1^2 \euler^{2 \beta \psi_2} \sum_{j \in Q_L}  \lVert F \rVert_{H^1 (S_3 (\x_j))}^2 + 2 C_1 \sum_{j \in Q_L}\lVert (\partial_{d+1} F)_0 \rVert_{L^2 (B (\x_j,\delta))}^2 .
\end{equation*}
Recall that $U_i (L) = \cup_{j \in Q_L} S_i (\x_j)$ and $W_\delta(L) = \cup_{j \in Q_L} B (\x_j , \delta)$. Hence, for all $\beta \geq \beta_1$ we have
\[
 \lVert F \rVert_{H^1 (U_1(L))}^2 \leq \tilde D_1 \lVert F \rVert_{H^1 (U_3 (L))}^2 +  \hat D_1 \lVert (\partial_{d+1} F)_0 \rVert_{L^2 (W_\delta(L))}^2 ,
\]
where
\begin{equation} \label{eq:D_tilde}
 \tilde D_1 (\beta) = 24 C_1 \cut_1^2 \euler^{2 \beta (\psi_2-\psi_1)}
 \quad \text{and} \quad
 \hat  D_1 (\beta)  = 2 C_1 \euler^{-2\beta \psi_1} .
\end{equation}
We choose $\beta$ such that
\begin{equation} \label{eq:beta}
 \euler^\beta = \left[ \frac{1}{12 \cut_1^2}  \frac{\lVert (\partial_{d+1} F)_0 \rVert_{L^2 (W_\delta(L)}^2}{\lVert F \rVert_{H^1 (U_3 (L))}^2} \right]^{\frac{1}{2\psi_2}} .
\end{equation}
Now we distinguish two cases. If $\beta \geq \beta_1$ we obtain by using $\psi_1 = 2 \psi_2$
\begin{equation} \label{eq:1}
 \lVert F \rVert_{H^1 (U_1(L))}^2 \leq 8 \sqrt{3} C_1 \cut_1 \lVert F \rVert_{H^1 (U_3 (L))}
 \lVert (\partial_{d+1} F)_0 \rVert_{L^2 (W_\delta(L))} .
\end{equation}
If $\beta < \beta_1$ we use Lemma~5.2 of \cite{RousseauL-12}. In particular, one concludes from Eq.~\eqref{eq:beta} that
\[
  \lVert F \rVert_{H^1 (U_3 (L))}^2  < \frac{1}{12\cut_1^2} \euler^{-2 \beta_1 \psi_2}  \lVert (\partial_{d+1} F)_0 \rVert_{L^2 (W_\delta(L))}^2.
\]
This gives us in the case $\beta < \beta_1$
\begin{equation}
 \lVert F \rVert_{H^1 (U_1(L))}^2 \leq \lVert F \rVert_{H^1 (U_3(L))}^{2}  %\nonumber \\[1ex]
 < \frac{\euler^{- \beta_1 \psi_2} }{\sqrt{12} \cut_1}    \lVert F \rVert_{H^1 (U_3(L))}
 \lVert (\partial_{d+1} F)_0 \rVert_{L^2 (W_\delta(L))} .\label{eq:2}
\end{equation}
If we set
\begin{equation} \label{eq:C2}
   D_1^{2} = \max \left\{ 8 \sqrt{3} C_1 \cut_1 , \frac{\euler^{-\beta_1 \psi_2}}{\cut_1 \sqrt{12}} \right\}
,
\end{equation}
we conclude the statement of the proposition from Ineqs.~\eqref{eq:1} and \eqref{eq:2}.
\end{proof}
Now we deduce from the second Carleman estimate, Proposition~\ref{prop:carleman2}, another interpolation inequality.
\begin{proposition} \label{prop:interpolation2}
For all $\delta \in (0,1/2)$, all $(1,\delta)$-equidistributed sequences $z_j$, all measurable and bounded $V: \RR^d \to \RR$, all $L \in \NN_{\mathrm{odd}}$, all $b \geq 0$ and all $\phi \in \mathrm{Ran} (\chi_{(-\infty,b]}(H_L))$
\begin{enumerate}[(a)]
 \item there is $\alpha_1 = \alpha_1 (d , \lVert V \rVert_\infty) \geq 1$ such that for all $\alpha \geq \alpha_1$ we have
 \begin{equation*}
\lVert F \rVert_{H^1 (X_1)}^2
\leq
\tilde D_2 (\alpha) \lVert F \rVert_{H^1 (U_1 (L))}^2
+
\hat D_2 (\alpha) \lVert F \rVert_{H^1 (\tilde X_{\rout_3})}
,
\end{equation*}
where $\alpha_1$ is given in Eq.~\eqref{eq:alpha1}, and $\tilde D_2 (\alpha)$ and $\hat D_2 (\alpha)$ are given in Eq.~\eqref{eq:D_hat};
 \item we have
 \[
\lVert F \rVert_{H^1 (X_1)}
\leq D_2 \lVert F \rVert_{H^1 (U_1 (L))}^{\gamma} \lVert F \rVert_{H^1 (\tilde X_{\rout_3})}^{1- \gamma} ,
\]
where $\gamma$ and $D_2$ are given in Eq.~\eqref{eq:gamma} and \eqref{eq:D2}.
\end{enumerate}
\end{proposition}
\begin{proof}
We choose a cutoff function $\chi \in C_{\mathrm c}^\infty (\RR^{d+1};[0,1])$ with $\supp \chi \subset B(\rout_3) \setminus \overline{B (r_1)}$,
$\chi (x) = 1$ if $x \in B(\rin_3) \setminus \overline{B(\rout_1)}$,
\[
 \max\{\lVert \Delta \chi \rVert_{\infty , V_1} , \lVert \lvert \nabla \chi \rvert \rVert_{\infty , V_1}\} \leq \frac{\tilde\cut_2}{\delta^4} =: \cut_2
\]
and
\[
 \max\{\lVert \Delta \chi \rVert_{\infty , V_3} , \lVert \lvert \nabla \chi \rvert \rVert_{\infty , V_3}\} \leq \cut_3 ,
\]
where $\tilde \cut_2$ depends only on the dimension and $\cut_3$ is an absolute constant\ifthenelse{\boolean{journal}}{}{, see Appendix~\ref{app:constants}}.
We set $u = \chi F$. We apply Proposition~\ref{prop:carleman2} with $\rho = \rout_3$ to the function $u$ and obtain for all $\alpha \geq \alpha_0 \geq 1$
 \[
\int_{B (\rout_3)} \left( \alpha {R_3^2} w^{1-2\alpha} \lvert \nabla u \rvert^2 + \alpha^3 w^{-1-2\alpha} \lvert u \rvert^2  \right) \drm x
\leq C_2 \rout_3^4 \int_{B (\rout_3)} w^{2-2\alpha} \lvert \Delta u \rvert^2 \drm x .
\]
Since $w \leq 1$ on $B(\rout_3)$ we can replace the exponent of the weight function $w$ at all three places by $2 - 2 \alpha$, i.e.
\begin{equation}\label{eq:I123}
 \int_{B (\rout_3)} \left( \alpha {R_3^2} w^{2-2\alpha} \lvert \nabla u \rvert^2 + \alpha^3 w^{2-2\alpha} \lvert u \rvert^2  \right) \drm x
\leq C_2 \rout_3^4 \int_{B (\rout_3)} w^{2-2\alpha} \lvert \Delta u \rvert^2 \drm x =: I .
\end{equation}
For the right hand side we use
\[
\Delta u = 2 (\nabla \chi)(\nabla F) + (\Delta \chi) F + (\Delta F) \chi,
\]
and $\Delta F = V_L F$, and obtain
\begin{align*}
I \leq 3 C_2 \rout_3^4 \int_{B (\rout_3)} \!\!\!\!\!\!\!\!\! w^{2-2\alpha} \left( 4\lvert (\nabla \chi)(\nabla F) \rvert^2 + \lvert (\Delta \chi) F  \rvert^2 + \lVert V \rVert_\infty^2 \lvert \chi F \rvert^2  \right) \drm x =: I_1 + I_2 + I_3.
\end{align*}
If we choose $\alpha$ sufficiently large, i.e.
\[
 \alpha \geq \left( 6 C_2 \rout_3^4 \lVert V \rVert_\infty^2 \right)^{1/3} =: \tilde \alpha_0 ,
\]
we can subsume the term $I_3$ into the left hand side of Ineq.~\eqref{eq:I123}. We obtain for all
\begin{equation} \label{eq:alpha1}
\alpha \geq \alpha_1 := \max\{\alpha_0 , \tilde\alpha_0\}
\end{equation}
the estimate
\[
 \int_{B (\rout_3)} \left( \alpha {R_3^2} w^{2-2\alpha} \lvert \nabla u \rvert^2 + \frac{\alpha^3}{2} w^{2-2\alpha} \lvert u \rvert^2  \right) \drm x
\leq I_1 + I_2.
\]
For the ``new'' left hand side we have the lower bound
\[
I_1 + I_2 \geq \int_{B (\rout_3)} \left( \alpha {R_3^2} w^{2-2\alpha} \lvert \nabla u \rvert^2 + \frac{\alpha^3}{2} w^{2-2\alpha} \lvert u \rvert^2  \right) \drm x  \geq \frac{1}{2} \left( \frac{\rout_3}{\rout_2} \right)^{2\alpha - 2} \lVert F \rVert_{H^1 (V_2)}^2 .
\]
For $I_1$ and $I_2$ we have the estimates
\[
I_1 \leq 3 C_2 \rout_3^4 \left[ 4 \cut_2^2 \left(\frac{\mathrm{e} \rout_3}{\rin_1}\right)^{2\alpha - 2}  \int_{V_1} \lvert \nabla F \rvert^2 + 4\cut_3^2 \left(\frac{\mathrm{e} \rout_3}{\rin_3}\right)^{2\alpha - 2}  \int_{V_3} \lvert \nabla F \rvert^2
\right]
\]
and
\[
I_2 \leq 3 C_2 \rout_3^4 \left[\cut_2^2 \left(\frac{\mathrm{e} \rout_3}{\rin_1}\right)^{2 \alpha - 2}  \int_{V_1} \lvert F \rvert^2 + \cut_3^2 \left(\frac{\mathrm{e} \rout_3}{\rin_3}\right)^{2 \alpha - 2}  \int_{V_3} \lvert F \rvert^2
\right] .
\]
Putting everything together, the Carleman estimate from Proposition~\ref{prop:carleman2} implies for $\alpha \geq \alpha_1$
\begin{equation} \label{eq:before_sum}
\lVert F \rVert_{H^1 (V_2)}^2 \leq 24 C_2 \rout_3^4
\left[ \cut_2^2 \left(\frac{\mathrm{e} \rout_2}{\rin_1}\right)^{2\alpha - 2} \!\!\! \lVert F \rVert_{H^1 (V_1)}^2
+ \cut_3^2 \left(\frac{\mathrm{e} \rout_2}{\rin_3}\right)^{2 \alpha - 2}\!\!\! \lVert F \rVert^2_{H^1 (V_3)} \right] .
\end{equation}
By translation, Ineq.~\eqref{eq:before_sum} is still true if we replace $V_1$, $V_2$ and $V_3$ by its translates $V_1 (\x_j)$, $V_2 (\x_j)$ and $V_3 (\x_j)$ for all $j \in Q_L$. Hence,
\begin{multline} \label{eq:summing}
\sum_{j \in Q_L} \lVert F \rVert_{H^1 (V_2 (\x_j))}^2
\leq 24 C_2 \rout_3^4
\left[\cut_2^2 \left(\frac{\mathrm{e} \rout_2}{\rin_1}\right)^{2\alpha - 2}
\sum_{j \in Q_L} \lVert F \rVert_{H^1 (V_1 (\x_j))}^2
\right. \\
\left.+ \cut_3^2  \left(\frac{\mathrm{e} \rout_2}{\rin_3}\right)^{2 \alpha - 2}
\sum_{j \in Q_L} \lVert F \rVert^2_{H^1 (V_3 (\x_j))} \right] .
\end{multline}
For all $L \in \NN_{\mathrm{odd}}$ Lemma~\ref{lemma:geometric} tells us that $\cup_{k \in Q_5}\cup_{j \in Q_L} V_2 (\x_j+kL) \supset X_1 = \Lambda_L \times [-1,1]$ and the left hand side is bounded from below by
\[
\sum_{j \in Q_L} \lVert F \rVert_{H^1 (V_2 (\x_j))}^2
=
\frac{1}{5^d} \sum_{k \in Q_5} \sum_{j \in Q_L}  \lVert F \rVert_{H^1 ( V_2 (\x_j + kL))}^2 \geq \frac{1}{5^d} \lVert F \rVert^2_{H^1 (X_1)} .
\]
Since $V_1 (\x_j) \cap \RR^{d+1}_+ \subset S_1 (\x_j)$, $S_1 (\x_i) \cap S_1 (\x_j) = \emptyset$ for $i \not = j$, and since $F$ is antisymmetric with respect to its last coordinate, we have
\[
\sum_{j \in Q_L} \lVert F \rVert_{H^1 (V_1 (\x_j))}^2 \leq
2  \sum_{j \in Q_L} \lVert F \rVert_{H^1 ( S_1 (\x_j))}^2
= 2 \lVert F \rVert_{H^1 (U_1 (L))}^2 .
\]
For the second summand on the right hand side of Ineq.~\eqref{eq:summing}, we find by Lemma~\ref{lemma:geometric} (iii) that there exists a constant $K_d$ such that
\[
\sum_{j \in Q_L} \lVert F \rVert^2_{H^1 (V_3 (\x_j))} \leq K_d \lVert F \rVert^2_{H^1 (\cup_{j \in Q_L} V_3 (\x_j))} .
\]
Moreover, since $\cup_{j \in Q_L} V_3 (\x_j) \subset \tilde X_{\rout_3} = \Lambda_{L + \rout_3} \times \left[-\rout_3 , \rout_3 \right]$, we have
\[
\sum_{j \in Q_L} \lVert F \rVert^2_{H^1 (V_3 (\x_j))} \leq
 K_d \lVert F \rVert_{H^1 (\tilde X_{\rout_3})} .
\]
Putting everything together we obtain for all $\alpha \geq \alpha_1$
\begin{equation} \label{eq:summing2}
\frac{1}{5^d} \lVert F \rVert_{H^1 (X_1)}^2
\leq
\tilde D_2 (\alpha) \lVert F \rVert_{H^1 (U_1 (L))}^2
+
\hat D_2 (\alpha) \lVert F \rVert_{H^1 (\tilde X_{\rout_3})}
,
\end{equation}
where
\begin{equation} \label{eq:D_hat}
 \tilde D_2 (\alpha) =  48 C_2 \rout_3^4 \cut_2^2  \left(\frac{\mathrm{e} \rout_2}{\rin_1}\right)^{2\alpha - 2}
 \ \text{and} \
 \hat D_2 (\alpha) =24 C_2 \rout_3^4 \cut_3^2 K_d \left(\frac{\mathrm{e} \rout_2}{\rin_3}\right)^{2 \alpha - 2} .
\end{equation}
If we let $c_1 = 48 C_2 \cut_2^2 \rout_3^4 \rin_1^2 / (\mathrm{e} \rout_2)^2$, $c_2 = 24 C_2 \cut_3^2 K_d \rout_3^4  \rin_3^2 / (\mathrm{e} \rout_2)^2$,
\[
p^+ = 2 \ln \left( \frac{\mathrm{e}\rout_2}{\rin_1} \right) > 0 \quad \text{and} \quad p^- = 2 \ln \left( \frac{\mathrm{e}\rout_2}{\rin_3} \right) < 0,
\]
then Ineq.~\eqref{eq:summing2} reads
\begin{equation} \label{eq:summing3}
\frac{1}{5^d}\lVert F \rVert_{H^1 (X_1)}^2
\leq c_1 \mathrm{e}^{p^+ \alpha} \lVert F \rVert_{H^1 (U_1 (L))}^2 +
c_2 \mathrm{e}^{p^- \alpha} \lVert F \rVert_{H^1 (\tilde X_{\rout_3})}^2  .
\end{equation}
We choose $\alpha$ such that
\begin{equation} \label{eq:alpha_choice}
\mathrm{e}^{\alpha} = \left( \frac{c_2}{c_1} \frac{\lVert F \rVert_{H^1 (\tilde X_{\rout_3})}^2}{ \lVert F \rVert_{H^1 (U_1 (L))}^2} \right)^{\frac{1}{p^+ - p^-}} .
\end{equation}
If $\alpha \geq \alpha_1$ we obtain from Ineq.~\eqref{eq:summing3} that
\begin{equation} \label{eq:3}
\frac{1}{5^d} \lVert F \rVert_{H^1 (X_1)}^2
\leq 2 c_1^{\gamma} c_2^{1 - \gamma} \lVert F \rVert_{H^1 (U_1 (L))}^{2\gamma} \lVert F \rVert_{H^1 (\tilde X_{\rout_3})}^{2- 2\gamma}, \quad \text{where} \quad \gamma = \frac{-p^-}{p^+ - p^-}  .
\end{equation}
If $\alpha < \alpha_1$, we proceed as in the last part of the proof of Proposition~\ref{prop:interpolation1}, i.e.\ we conclude from Eq.~\eqref{eq:alpha_choice} that
\[
 \lVert F \rVert_{H^1 (\tilde X_{\rout_3})}^2 < \frac{c_1}{c_2} \mathrm{e}^{\alpha_1 (p^+ - p^-)} \lVert F \rVert_{H^1 (U_1 (L))}^2
\]
and thus
\begin{equation} \label{eq:4}
\lVert F \rVert_{H^1 (X_1)}^2 \leq
\lVert F \rVert_{H^1 (\tilde X_{\rout_3})}^{2 \frac{p^+ - p^-}{p^+ - p^-}} <
 \lVert F \rVert_{H^1 (\tilde X_{\rout_3})}^{2 (1-\gamma)}
\left( \frac{c_1}{c_2} \mathrm{e}^{\alpha_1 (p^+ - p^-)} \right)^{\gamma}
\lVert F \rVert_{H^1 (U_1 (L))}^{2 \gamma} .
\end{equation}
We calculate
\begin{equation} \label{eq:gamma}
\gamma = \frac{\ln 2}{\ln (\rin_3 / \rin_1)},
\end{equation}
set
\begin{equation} \label{eq:D2}
 D_2^2 = \max\left\{5^d 192\cdot 9^4 C_2 \cut_3^2 K_d \mathrm{e}^4 d^2\left(\frac{2 \cut_2^2 \rin_1^2}{\cut_3^2 K_d \rin_3^2}\right)^\gamma \, , \, \left( \frac{2 \cut_2^2}{\cut_3^2 K_d} \left( \frac{\rin_3}{\rin_1} \right)^{2(\alpha_1-1)} \right)^\gamma  \right\}
\end{equation}
and conclude the statement of the proposition from Ineqs.~\eqref{eq:3} and \eqref{eq:4}.
\end{proof}
%
%%%%%%%%%%%%
%%%
%%% Proof of Theorem
%%%
%%%%%%%%%%%%
%
\subsection{Proof of Theorem~\ref{thm:result1} and Corollary~\ref{cor:eigenvalue}} \label{section:proof_main_result}
\begin{proposition} \label{prop:upper_lower}
For all $T > 0$, all measurable and bounded $V : \RR^d \to \RR$, all $L \in \NN_{\mathrm{odd}}$, all $b \geq 0$ and all $\phi \in \mathrm{Ran} (\chi_{(-\infty,b]}(H_L))$ we have
\[
\frac{T}{2} \!\!\!\sum_{\genfrac{}{}{0pt}{2}{k \in \NN}{E_k \leq b}} \lvert \alpha_k \rvert^2 \leq \frac{\lVert F \rVert_{H^1 (\Lambda_{\R L} \times [-T,T])}^2}{R^d}
\leq 2 T (1 + (1+\lVert V \rVert_\infty)T^2) \!\!\! \sum_{\genfrac{}{}{0pt}{2}{k \in \NN}{E_k \leq b}} \!\!\! \beta_k(T)  \lvert \alpha_k \rvert^2 ,
\]
where
\[
 \beta_k(T) = \begin{cases}
	      1 & \text{if}\ E_k \leq 0,  \\
	      \mathrm{e}^{2T \sqrt{E_k} } & \text{if} \ E_k > 0 . \\
           \end{cases}
\]
\end{proposition}
\begin{proof}
For the function $F : \Lambda_{\R L} \times \RR \to \CC$ we have for $T > 0$
\[
\lVert F \rVert_{H^1 (\Lambda_{\R L} \times [-T,T])}^2
= \int_{-T}^T \int_{\Lambda_{\R L}} \left( 
\lvert \partial_{d+1} F \rvert^2 +
\lvert \nabla' F \rvert^2+
\lvert F \rvert^2 
\right) \drm x .
\] 
Note that $\lVert \phi_k \rVert_{L^2 (\Lambda_{\R L})} = \R^d$. By Green's theorem we have
\[
  \int_{\Lambda_{\R L}} \lvert \nabla' F \rvert^2 \drm x' = \int_{\Lambda_{\R L}} (-\sum_{i=1}^d \partial^2_i F) \overline{F} \drm x' = - \int_{\Lambda_{\R L}} V\lvert F\rvert^2 \drm x' + 
 \int_{\Lambda_{\R L}} (\partial^2_{d+1} F) \overline{F} \drm x' 
\]
for all $x_{d+1} \in \RR$. 
First we estimate  
\begin{align*}
 \lVert F \rVert_{H^1 (\Lambda_{\R L} \times [-T,T])}^2
 &=
 \int_{-T}^T \int_{\Lambda_{\R L}} \left( 
 \lvert \partial_{d+1} F \rvert^2  - V\lvert F\rvert^2 + (\partial^2_{d+1} F) \overline{F}
 + \lvert F \rvert^2 \right) \drm x \\
 &\leq
 \int_{-T}^T \int_{\Lambda_{\R L}} \left( 
\lvert \partial_{d+1} F \rvert^2  + (\partial^2_{d+1} F) \overline{F}
+ (1+ \lVert V \rVert_\infty) \lvert F \rvert^2 \right) \drm x \\[1ex]
&=  2\R^d  \sum_{\genfrac{}{}{0pt}{2}{k \in \NN}{E_k \leq b}} \lvert \alpha_k\rvert^2  I_k ,
\end{align*}
where
\begin{align*}
  I_k 
 &:= \int_{0}^T \left( (1+\lVert V \rVert_\infty)\funs_k(x_{d+1})^2 + \funs'_k(x_{d+1})^2 + \funs''_k(x_{d+1})\funs_k(x_{d+1}) \right) \drm x_{d+1} \\
 &= (1+\lVert V \rVert_\infty) \int_0^T  \funs_k(x_{d+1})^2 \drm x_{d+1} + \funs'_k(T)\funs_k(T).
\end{align*}
If $E_k \leq 0$, we estimate using $s_k(t) \leq t$ and $s_k'(t) s_k(t) \leq t$ for $t > 0$
\[
 I_k \leq (1 + \lVert V \rVert_\infty) T^3/3 + T 
 \leq 
 ((1 + \lVert V \rVert_\infty) T^3 + T ) \beta_k(T).
\]
For $E_k > 0$ we use $\sinh(\omega_k t)/\omega_k \leq t \cosh(\omega_k t)$ for $t > 0$ and $\cosh(\omega_k T)^2 \leq \euler^{2 \omega_k T}$ to obtain
\begin{align*}
 I_k &= (1+\lVert V \rVert_\infty) \int_0^T \frac{\sinh^2(\omega_kx_{d+1})}{\omega_k^2} \drm x_{d+1} + \sinh(\omega_k T) \cosh(\omega_k T)/\omega_k\\
&\leq ((1 + \lVert V \rVert_\infty ) T^3 \cosh^2(\omega_k T) + T \cosh^2(\omega_k T) ) \leq ((1 + \lVert V \rVert_\infty ) T^3 + T) \beta_k(T).
\end{align*}
This shows the upper bound.
For the lower bound we drop the gradient term and obtain
 \begin{align*}
\lVert F \rVert_{H^1 (\Lambda_{\R L} \times [-T,T])}^2 &\ge 
\int_{-T}^T \int_{\Lambda_{\R L}} \left( 
\lvert \partial_{d+1} F \rvert^2 +
\lvert F \rvert^2 
\right) \drm x = 
2 \cdot \R^d \sum_{\genfrac{}{}{0pt}{2}{k \in \NN}{E_k \leq b}} \lvert \alpha_k \rvert^2 \tilde I_k ,
\end{align*}  
where 
\[
 \tilde I_k := \int_{0}^T\left[ \funs_k(x_{d+1})^2+\funs_k'(x_{d+1})^2 \right] \drm x_{d+1} .
\]
If $E_k = 0$, the lower bound $\tilde I_k \geq T$ follows immediately.
Else, we have $\funs_k(t)^2 \geq \sin^2(\omega_k t)/\omega_k$ and $\funs'_k(t)^2 \geq \cos(\omega_k t)$ whence
\[
 \tilde I_k \geq 
 \int_0^T \frac{\sin^2(\omega_k x_{d+1})}{\omega_k^2} + \cos^2(\omega_k x_{d+1}) \drm x_{d+1}
 \geq
 \int_0^T \cos^2(\omega_k x_{d+1}) \drm x_{d+1}
 =
 \frac{T}{2} + \frac{\sin(2 \omega_k T)}{4 \omega_k}.
\]
Now, if $2 \omega_k T < \pi$, the sinus term is positive and we drop it to find $ \tilde I_k \geq 
 T/2$.
If $2 \omega_k T \geq \pi$, we have $\sin(2 \omega_k T ) \geq -1$ and estimate
\[
 \tilde I_k 
 \geq
 \frac{T}{2} - \frac{1}{4 \omega_k} 
 =
 \frac{T}{2} - \frac{\pi}{4 \pi\omega_k} 
 \geq
 \frac{T}{2} - \frac{T}{2 \pi} 
 \geq
 \frac{T}{4}.
 \qedhere
\]
\end{proof}
\begin{proof}[Proof of Theorem~\ref{thm:result1}]
First we consider the case $L \in \NN_{\mathrm{odd}}$. We note that Proposition~\ref{prop:upper_lower} remains true if we replace $\Lambda_{\R L}$ by $\Lambda_{L}$ and $\R^d$ by 1, i.e.\ for all $T > 0$ and $L \in \NN_{\mathrm{odd}}$ we have
\begin{equation}\label{eq:lower}
\frac{T}{2} \!\!\! \sum_{\genfrac{}{}{0pt}{2}{k \in \NN}{E_k \leq b}} \lvert \alpha_k \rvert^2 \leq \lVert F \rVert_{H^1 (\Lambda_{L} \times [-T,T])}^2
\leq 2 T (1 + (1+\lVert V \rVert_\infty)T^2) \!\!\! \sum_{\genfrac{}{}{0pt}{2}{k \in \NN}{E_k \leq b}} \!\!\! \beta_k(T)  \lvert \alpha_k \rvert^2 .
\end{equation}
We have $\tilde X_{\rout_3} \subset \Lambda_{\R L} \times [-R_3 , R_3]$. By Ineq.~\eqref{eq:lower} and Proposition~\ref{prop:upper_lower} we have
\begin{align*}
\frac{\lVert F \rVert_{H^1 (\tilde X_{\rout_3})}^2 }{\lVert F \rVert_{H^1 (X_1)}^2 }
\leq
\frac{\lVert F \rVert_{H^1 (\Lambda_{\R L} \times [-\rout_3, \rout_3])}^2 }{\lVert F \rVert_{H^1 (X_1)}^2}
\leq \tilde D_3^2 D_4^2
\end{align*}
with
\[
 \tilde D_3^2 = \frac{\sum_{E_k \leq b} \theta_k  \lvert \alpha_k \rvert^2}{\sum_{E_k \leq b}  \lvert \alpha_k \rvert^2} \quad \text{and} \quad D_4^2 = 4\cdot \R^d \rout_3 (1 + (1+\lVert V \rVert_\infty)\rout_3^2),
\]
where $\theta_k = \beta_k (\rout_3)$.
We use Propositions~\ref{prop:interpolation1} and \ref{prop:interpolation2} and obtain
\begin{align*}
\lVert F \rVert_{H^1 (\tilde X_{\rout_3})} &\leq \tilde D_3 D_4 \lVert F \rVert_{H^1 (X_1)} 
\leq D_1^\gamma D_2 \tilde D_3 D_4
\lVert F \rVert_{H^1 (\tilde X_{\rout_3})}^{1 - \gamma}
\lVert (\partial_{d+1} F)_0 \rVert_{L^2 (W_\delta(L))}^{\gamma / 2}
\lVert F \rVert_{L^2 (U_3(L))}^{\gamma / 2} .
\end{align*}
Since $U_3 (L) \subset \tilde X_{\rout_3}$ we have
\[
\lVert F \rVert_{H^1 (\tilde X_{\rout_3})} \leq
 D_1^2 D_2^{2/\gamma} \tilde D_3^{2/\gamma} D_4^{2/\gamma}
\lVert (\partial_{d+1} F)_0 \rVert_{L^2 (W_\delta(L))} .
\]
By Ineq.~\eqref{eq:lower}, the square of the left hand side is bounded from below by
\[
\lVert F \rVert_{H^1 (\tilde X_{\rout_3})}^2 \geq
\lVert F \rVert_{H^1 (\Lambda_L \times [-\rout_3 , \rout_3 ] ) }^2 \geq
 \frac{\rout_3}{2}\sum_{\genfrac{}{}{0pt}{2}{k \in \NN}{E_k \leq b}} \lvert \alpha_k \rvert^2.
\]
 Putting everything together we obtain by using $(\partial_{d+1} F)_0 = \phi$
 \begin{equation*}
 \frac{\rout_3}{2}\sum_{\genfrac{}{}{0pt}{2}{k \in \NN}{E_k \leq b}} \lvert \alpha_k \rvert^2  \leq
 D_1^4 \left( D_2 \tilde D_3 D_4 \right)^{4/\gamma} \lVert \phi \rVert_{L^2 (W_\delta (L))}^2 .
\end{equation*}
In order to end the proof we will give an upper bound on $\tilde D_3$ which is independent of $\alpha_k$, $k \in \NN$.
For this purpose, we we recall that $\theta_k = \beta_k (\rout_3)$. Since
$
\theta_k \leq \euler^{2 \rout_3 \sqrt{b}}
$
for all $k \in \NN$ with $E_k \leq b$, we have
\[
 \tilde D_3^4 \leq D_3^4 :=
   \euler^{4 \rout_3 \sqrt{b} }.
\]
Hence, using $\sum_{E_k \leq b} \lvert \alpha_k \rvert^2 = \lVert \phi \rVert_{L^2(\Lambda_L)}^2$, we obtain for all $L \in \NN_{\mathrm{odd}}$ the estimate
\begin{equation*} \label{ucp1}
  \tilde C_{\sfuc} \lVert \phi \rVert_{L^2 (\Lambda_L)}^2 \leq  \lVert \phi \rVert_{L^2 (W_\delta(L))}^2
\end{equation*}
where $\tilde C_{\sfuc} = \tilde C_{\sfuc} (d,\delta , b , \lVert V \rVert_\infty) =  D_1^{-4} \left( D_2 D_3 D_4 \right)^{-4/\gamma}$. From the definitions of $D_i$, $i \in \{1,2,3,4\}$, and $\gamma$ one calculates that
\[
\tilde C_{\sfuc} \geq \delta^{\tilde N \bigl(1 + \lVert V \rVert_\infty^{2/3} + \sqrt{b} \bigr)}
\]
with some constant $\tilde N = \tilde N (d)$\ifthenelse{\boolean{journal}}{}{, see Appendix~\ref{app:constants}}.
Now we treat the case of $L \in \NN_{\mathrm{even}} = \{2,4,6,\ldots\}$. By a scaling argument as in Corollary~2.2 of \cite{RojasMolinaV-13}, we immediately obtain that for all $G > 0$, $\delta \in (0,G/2)$, $L/G \in \NN_{\mathrm{odd}}$ and all $(G,\delta)$-equidistributed sequences $q_j$
we have
\begin{equation} \label{eq:prefinal}
 \lVert \phi \rVert_{L^2 (W_\delta^q (L))}^2
\geq \tilde C_{\sfuc}^{G} \lVert \phi \rVert_{L^2 (\Lambda_L)}^2
\end{equation}
and $\tilde C_{\sfuc}^{G} (d, \delta , b ,\lVert V \rVert_\infty) =  \tilde C_{\sfuc} (d, \delta / G , b G^2 , \lVert V \rVert_\infty G^2)$. Here $W_\delta^q (L)$ denotes the set $W_\delta (L)$ corresponding to the sequence $q_j$.
Now we define
\[
G = \begin{cases}
\frac{L}{L / 2 - 1} & \text{if}\ L \in 4 \NN, \\
2 & \text{otherwise}
\end{cases}
\]
which satisfies  $G \in [2,4]$ and $L/G \in \NN_{\mathrm{odd}}$.
Since $G \geq 2$, every elementary cell $\Lambda_G + j$, $j \in (G\ZZ)^d$ contains at least one elementary cell $\Lambda_1 + j$, $j \in \ZZ^d$. Hence we can choose a $(G,\delta)$-equidistributed subsequence $q_j$ of $\x_j$. We apply Ineq.~\eqref{eq:prefinal} to this subsequence and obtain
\begin{equation*}
 \lVert \phi \rVert_{L^2 (W_\delta(L))}^2
\geq
 \lVert \phi \rVert_{L^2 (W^q_\delta(L))}^2
 \geq \tilde C_\sfuc^{G} \lVert \phi \rVert_{L^2 (\Lambda_L)} .
\end{equation*}
Note that $W_\delta (L)$ corresponds to the sequence $\x_j$.
Putting everything together we obtain the statement of the theorem with
\[
\min \Bigl\{ \tilde C_\sfuc , \inf_{G \in [2,4]} \tilde C^{G}_\sfuc \Bigr\}
        \geq \delta^{N \bigl(1 + \lVert V \rVert_\infty^{2/3} + \sqrt{b} \bigr)} =: C_\sfuc
\]
and some constant $N = N (d)$. For the last inequality we use that $(1/4)^{\tilde N} \geq \delta^{2 \tilde N}$.
\end{proof}

\begin{proof}[Proof of Corollary~\ref{cor:eigenvalue}]
We denote the normalized eigenfunctions of $-\Delta_L + A_L + B_L$ corresponding to the eigenvalues $\lambda_i(-\Delta_L + A_L + B_L)$ by $\phi_i$. Then we have
\begin{align*}
  \lambda_i( - \Delta_L + A_L + B_L)
    & = \left\langle \phi_i, ( - \Delta_L + A_L + B_L ) \phi_i \right\rangle\\
    & = \max_{\phi \in \mathrm{Span}\{\phi_1, \ldots, \phi_i\}, \lVert \phi \rVert = 1}  \left\langle \phi, (- \Delta_L + A_L) \phi \right\rangle + \left\langle \phi ,  B_L \phi \right\rangle\\
    & \geq \max_{\phi \in \mathrm{Span}\{\phi_1, \ldots, \phi_i\}, \lVert \phi \rVert = 1} \left\langle \phi, (- \Delta_L + A_L) \phi \right\rangle + \alpha \left\langle \phi , \chi_{W_{\delta}(L)} \phi \right\rangle.
 \end{align*}
By Corollary~\ref{cor:result1}, we conclude that for all $\phi \in \mathrm{Span}\{\phi_1, \ldots, \phi_i\}$, $\lVert \phi \rVert = 1$, we have
  \[
    \left\langle \phi , \chi_{W_{\delta}(L)} \phi \right\rangle \geq C_\sfuc^{G,1}(d,\delta,b,\lVert A_L + B_L \rVert_\infty )
  \]
and furthermore, by the variational characterization of eigenvalues, we find
  \begin{equation*}
     \max_{\genfrac{}{}{0pt}{1}{\phi \in \mathrm{Span}\{\phi_1, \ldots, \phi_i\}}{\lVert \phi \rVert = 1}} \left\langle \phi, (- \Delta_L +  A_L) \phi \right\rangle
     \geq \inf_{\mathrm{dim} \mathcal{D} = i} \max_{\genfrac{}{}{0pt}{1}{\phi \in \mathcal{D}}{\lVert \phi \rVert = 1}} \left\langle \phi, (- \Delta_L + A_L) \phi \right\rangle
     = \lambda_i( - \Delta_L + A_L) .
  \end{equation*}
Thus, we obtain the statement of the corollary.
\end{proof}
%
%%%%%%%%%%%%
%%%
%%% Proof of Wegner and initial scale estimate
%%%
%%%%%%%%%%%%
%
\section{Proof of Wegner and initial scale estimate}\label{s:proof-Wegner}

Recall that $0 < G_1 < G_2$ are the numbers from the Delone property such that $\lvert \{ \mathcal{D} \cap (\Lambda_{G_1} + x) \} \rvert \leq 1$, $\lvert \{ \mathcal{D} \cap (\Lambda_{G_2} + x) \} \rvert \geq 1$ for any $x \in \RR^d$, and that for all $t \in [0,1]$ we have $\supp u_t \subset \Lambda_{G_u}$.
Let $\delta_{\max} := 1 - \omega_{+}$ and $K_u := u_{\max}  \lceil G_u/G_1 \rceil^d$.
For $\omega \in [\omega_{-}, \omega_{+}]^{\mathcal{D}}$ and $\delta \leq \delta_{\max}$, we use the notation $V_{\omega + \delta}$ for the potential $V_\omega$, where every $\omega_j$, $j \in \mathcal{D}$ has been replaced by $\omega_j + \delta$.
The following lemma is a consequence of the properties of a Delone set, in particular $\lvert \Lambda_L \cap \mathcal{D} \rvert \leq \lceil L/G_1 \rceil^d$, and our assumption \eqref{eq:condition_u}.
\begin{lemma}
\label{lem:wegner}
 \begin{enumerate}[(i)]
  \item
  For all $\omega \in [\omega_{-}, \omega_{+}]^{\mathcal{D}}$, all $0 < \delta \leq \delta_{\max}$ and all $L \in (G_2 + G_u) \NN$, the difference $V_{\omega + \delta} - V_\omega$ is on $\Lambda_L$ bounded from below by $\au \delta^{\ao}$ times the characteristic function of $W_{\bu \delta^{\bo}}(L)$ which corresponds to a $(G_2 + G_u, \bu \delta^{\bo})$-equidistributed sequence.
  \item
  For all $\omega \in [0,1]^{\mathcal{D}}$ we have $\lVert V_\omega \rVert_\infty \leq K_u$.
  \item
  For all $L \in (G_2 + G_u) \NN$, we have $\lvert \{ j \in \mathcal{D} : \exists t \in [0,1]: \supp u_t( \cdot - j) \cap \Lambda_L \neq \emptyset \} \rvert \leq  \lceil (L + G_u) / G_1 \rceil^d \leq (2 L /G_1)^d$.
 \end{enumerate}
\end{lemma}
\begin{proof}[Proof of Theorem~\ref{thm:wegner}]
Note that for all $b \in \RR$, $\lambda_i(H_{\omega, L}) \leq b$ implies, by Lemma~\ref{lem:wegner} part (ii), that $\lambda_i(H_{\omega + \delta,L}) \leq b + \lVert V_{\omega + \delta} - V_\omega \rVert \leq b + 2 K_u$.
Now we apply Corollary~\ref{cor:eigenvalue} with $A_L = V_\omega$ and $B_L = V_{\omega + \delta} - V_\omega$ (both restricted to $\Lambda_L$).
Together with Lemma \ref{lem:wegner} part (i), we obtain for all $b \in \RR$, all $L \in G_u \NN$, all $\omega \in [\omega_{-}, \omega_{+}]^{\mathcal{D}}$, all $\delta \leq \delta_{\max}$ and all $i \in \NN$ with $\lambda_i(H_{\omega, L}) \leq b$ the inequality
\[
   \lambda_i(H_{\omega + \delta, L})
   \geq
   \lambda_i(H_{\omega, L}) + \au \delta^{\ao} C_\sfuc^{G_2 + G_u,1} ( d, \bu \delta^{\bo}, b + 2 K_u, K_u) .
\]
In particular, there is $\kappa = \kappa(d, \omega_+ , \au , \ao , \bu , \bo , G_2 , G_u , K_u , b) > 0$ such that
\begin{equation} \label{eq:eigenvalue_lifting_kappa}
 \lambda_i(H_{\omega + \delta, L})
   \geq
   \lambda_i(H_{\omega, L}) + \delta^\kappa .
\end{equation}
Now let $\epsilon > 0$, satisfying $\epsilon \leq \epsilon_{\max} := \delta_{\max}^\kappa/4$.
We choose $\delta := (4 \epsilon)^{1/\kappa}$, whence
\begin{equation}\label{eq:eigmove}
 \lambda_i(H_{\omega + \delta,L}) \geq \lambda_i(H_{\omega,L}) + 4 \epsilon .
\end{equation}
Let $\rho \in C^\infty(\mathbb{R},[-1,0])$ be smooth, non-decreasing such that $\rho = -1$ on $(-\infty; -\epsilon]$ and $\rho = 0$ on $[\epsilon; \infty)$. We can assume $\lVert \rho' \rVert_\infty \leq 1/\epsilon$. It holds that
\begin{equation*}
\chi_{[E-\epsilon; E + \epsilon]} (x) \leq \rho(x-E + 2\epsilon) - \rho(x-E-2\epsilon)
= \rho(x-E - 2\epsilon + 4 \epsilon) - \rho(x-E-2\epsilon)
\end{equation*}
for all $x \in \mathbb{R}$ and together with \eqref{eq:eigmove} this implies
\begin{align}
 \mathbb{E} \left[ \mathrm{Tr} \left[ \chi_{[E-\epsilon; E+\epsilon]} ( H_{\omega,L}) \right] \right]
 &\leq
 \mathbb{E} \left[ \mathrm{Tr} \left[ \rho(H_{\omega,L} - E - 2 \epsilon + 4 \epsilon) - \rho(H_{\omega,L} - E - 2\epsilon) \right] \right] \nonumber \\
 & \leq \mathbb{E} \left[ \mathrm{Tr} \left[ \rho \left( H_{\omega + \delta,L} - E - 2 \epsilon \right) - \rho \left( H_{\omega,L} - E - 2 \epsilon \right) \right] \right]. \label{trace rho}
 \end{align}
Now let $\tilde \Lambda_L := \{ j \in \mathcal{D} : \exists t \in [0,1]: \supp u_t( \cdot - j) \cap \Lambda_L \neq \emptyset \}$ be the set of lattice sites which can influence the potential within $\Lambda_L$.
Note that $\lvert \tilde \Lambda_L \rvert \leq ( 2 L / G_1 )^d$.
We enumerate the points in $\tilde \Lambda_L$ by
$k : \{ 1, \ldots \lvert \tilde \Lambda_L \rvert \} \rightarrow \mathcal{D}$, $n \mapsto k(n)$.
The upper bound in \eqref{trace rho} will be expanded in a telescopic sum by changing the $\lvert \tilde \Lambda_L \rvert$ indices from $\omega_j$ to $\omega_j + \delta$ successively.
In order to do that some notation is needed. Given
$\omega \in [\omega_{-}, \omega_{+}]^{\mathcal{D}}$,
$n \in \{ 1, \ldots, \lvert \tilde \Lambda_L \rvert \}$,
$\delta \in [0, \delta_{\max}]$ and
$t \in [\omega_{-},  \omega_{+}]$,
we define $\tilde{\omega}^{(n, \delta)}(t) \in [\omega_{-}, 1]^{\mathcal{D}}$ inductively via
\begin{align*}
\left( \tilde{\omega}^{(1, \delta)}(t) \right)_j &:=
	\begin{cases}
		 t & \mbox{ if } j = k(1),\\
		\omega_j & \mbox{else},
	\end{cases} \quad\text{and}\quad
\left( \tilde{\omega}^{(n, \delta)}(t) \right)_j :=
	\begin{cases}
		 t & \mbox{ if } j = k(n),\\
		\left( \tilde{\omega}^{(n-1, \delta)}(\omega_j + \delta) \right)_j & \mbox{else}.
	\end{cases}
\end{align*}
The function
$\tilde{\omega}^{(n,\delta)} : [\omega_{-},  1] \rightarrow [\omega_{-},  1]^{\mathcal{D}}$
is the rank-one perturbation of $\omega$ in the $k(n)$-th coordinate with the additional requirement that all sites $k(1), \ldots, k(n-1)$ have already been blown up by $\delta$.
We define
\begin{align*}
\Theta_n(t) &:= \mathrm{Tr} \left[ \rho \left( H_{{\tilde{\omega}^{(n, \delta)}(t)},L} - E - 2 \epsilon \right) \right],
\mbox{ for } n = 1, \ldots, \lvert \tilde \Lambda_L \rvert.
\end{align*}
Note that
\begin{align*}
\Theta_1(\omega_{k(1)}) & = \mathrm{Tr} \left[ \rho \left( H_{\omega, L} - E - 2 \epsilon \right) \right],\\
\Theta_n(\omega_{k(n)}) &= \Theta_{n-1}(\omega_{k(n-1)} + \delta)\quad \mbox{for}\ n = 2, \ldots, \lvert \tilde \Lambda_L \rvert \quad \text{and}\\
\Theta_{\lvert \tilde \Lambda_L \rvert}(\omega_{k(\lvert \tilde \Lambda_L \rvert)} + \delta) & = \mathrm{Tr} \left[ \rho \left( H_{\omega + \delta, L} - E - 2 \epsilon \right) \right].
\end{align*}
Hence the upper bound in \eqref{trace rho} is
\begin{multline*}
 \EE \left[ \mathrm{Tr} \left[ \rho(H_{\omega + \delta,L} - E - 2 \epsilon ) \right]- \mathrm{Tr} \left[ \rho(H_{\omega,L} - E - 2 \epsilon  \right]  \right] \\
= \EE \left[ \Theta_{\lvert \tilde \Lambda_L \rvert}(\omega_{k(\lvert \tilde \Lambda_L \rvert)} + \delta) - \Theta_1(\omega_{k(1)}) \right]
=  \sum_{n = 1}^{\lvert \tilde \Lambda_L \rvert} \EE \left[ \Theta_n(\omega_{k(n)} + \delta) - \Theta_n(\omega_{k(n)}) \right].
\end{multline*}
Due to the product structure of the probability space, we can apply Fubini's Theorem to each summand and obtain
\begin{align*}
\EE \left[ \Theta_n(\omega_{k(n)} + \delta) - \Theta_n(\omega_{k(n)}) \right]
= \EE \left[ \int_{\omega_{-}}^{\omega_{+}} \Theta_n(\omega_{k(n)} + \delta) -\Theta_n(\omega_{k(n)})  \mathrm{d}\mu(\omega_{k(n)})\right].
\end{align*}
Note that $\Theta_n:\ [\omega_{-}, 1] \to \RR$ is monotone and bounded.
We will use the following Lemma.
	\begin{lemma}
	\label{lemma:Wegner1}
	Let $- \infty < \omega_{-} < \omega_{+} \leq + \infty$.
	Assume that $\mu$ is a probability distribution with bounded density $\nu_\mu$ and support in the interval $[\omega_{-},\omega_{+}]$ and let $\Theta$ be a non-decreasing, bounded function.
	Then for all $\delta > 0$
	\begin{equation*}
	\int_\RR \left[ \Theta(\lambda + \delta) - \Theta(\lambda) \right] \mathrm{d} \mu(\lambda)  \leq \lVert \nu_\mu \rVert_\infty \cdot \delta \left[ \Theta(\omega_{+} + \delta) - \Theta(\omega_{-}) \right].
	\end{equation*}
	\end{lemma}
		\begin{proof}[Proof of Lemma \ref{lemma:Wegner1}]
		We calculate
		\begin{align*}
		& \int_\RR \left[ \Theta(\lambda + \delta) - \Theta (\lambda) \right] \mathrm{d} \mu ( \lambda )\\
		 \leq & \lVert \nu_\mu \rVert_\infty \int_{\omega_{-}}^{\omega_{+}} \left[ \Theta(\lambda + \delta) - \Theta (\lambda) \right] \mathrm{d} \lambda
		=   \lVert \nu_\mu \rVert_\infty \left[ \int_{\omega_{-} + \delta}^{\omega_{+}+\delta} \Theta(\lambda) \mathrm{d} \lambda - \int_{\omega_{-}}^{\omega_{+}} \Theta(\lambda) \mathrm{d} \lambda \right]\\
		=&  \lVert \nu_\mu \rVert_\infty \left[ \int_{\omega_{+}}^{\omega_{+}+\delta} \Theta(\lambda) \mathrm{d} \lambda - \int_{\omega_{-}}^{\omega_{-}+\delta} \Theta(\lambda) \mathrm{d} \lambda \right]
		\leq   \lVert \nu_\mu \rVert_\infty \cdot \delta \left[ \Theta(\omega_{+}+\delta) - \Theta(\omega_{-}) \right]. \qedhere
		\end{align*}
		\end{proof}
Thus, we find for all $n = 1,\ldots, \lvert \tilde \Lambda_L \rvert$
\begin{equation*}
\int_{\omega_{-}}^{\omega_{+}} \left[ \Theta_n(\omega_{k(n)} +  \delta) - \Theta_n(\omega_{k(n)})  \mathrm{d}\mu(\omega_{k(n)})\right] \leq \lVert \nu_\mu \rVert_{\infty} \cdot \delta \left[ \Theta_n(\omega_{+}+  \delta) - \Theta_n(\omega_{-}) \right] .
\end{equation*}
We will also need the following result, see, e.g., Theorem~2 in \cite{HundertmarkKNSV-06}.
\begin{proposition}\label{Krein}
	Let $H_0 := -\Delta + A$ be a Schr\"o\-ding\-er operator with a bounded potential $A\geq 0$, and
	let $H_1 := H_0 + B$ for some bounded $B \geq 0$ with compact support.	
	Denote the corresponding Dirichlet restrictions to $\Lambda$ by $H_0^\Lambda$ and $H_1^\Lambda$, respectively.
	There are constants $K_1$, $K_2$ depending only on $d$ and monotonously on $\mathrm{diam}\ \supp B$
	such that for any smooth, bounded function $g:~\RR \rightarrow \RR$ with compact support in $(-\infty,b]$
	and the property that $g(H_1^\Lambda) - g(H_0^\Lambda)$ is trace class we have
	\begin{equation*}
	\mathrm{Tr} \left[ g(H_1^\Lambda) - g(H_0^\Lambda) \right]
% =\int h(\lambda) \xi(\lambda) \mathrm{d} \lambda
\leq K_1 \euler^{b} + K_2 \left( \ln(1+\lVert g^\prime \rVert_\infty )^d \right) \lVert g^\prime \rVert_1 .
	\end{equation*}
	\end{proposition}
Proposition~\ref{Krein} implies
	\begin{lemma}\label{lemma:SSF}
	Let $0 < \epsilon \leq \epsilon_{\max}$.
	Then $\Theta_n(\omega_{+} + \delta) - \Theta_n(\omega_{-}) \leq ( K_1 \euler^{b} + 2^d K_2 ) \lvert \ln \epsilon \rvert^d$, where $K_1, K_2$ are as in Proposition~\ref{Krein} and thus only depend on $d$ and on $G_u$.
	\end{lemma}
		\begin{proof}[Proof of Lemma \ref{lemma:SSF}]
		Let $g(\cdot) := \rho(\cdot - E - 2\epsilon))$.
		By our choice of $\rho$, $g$ has support in $(- \infty, b]$, $\lVert g^\prime\rVert_\infty \leq 1/\epsilon$ and $\lVert g^\prime \rVert_1 = 1$.
		We define the operators
		\begin{align*}
		H_0^\Lambda  := H \left( {\tilde{\omega}^{(n,\delta)}(\omega_{-})} ,L \right) \quad \text{and} \quad
		H_1^\Lambda  := H \left( {\tilde{\omega}^{(n,\delta)}(\omega_{+} + \delta)} ,L \right).
		\end{align*}
		They are lower semibounded operators with purely discrete spectrum and since $g$ has support in $(- \infty, b]$, the difference $g(H_1^\Lambda)-g(H_0^\Lambda)$ is trace class.
		By the previous proposition
		\begin{equation*}
		\Theta_n(\omega_{+} + \delta) - \Theta_n(\omega_{-}) = \mathrm{Tr} \left[ g(H_1^\Lambda) - g(H_0^\Lambda) \right] \leq K_1 \euler^b + K_2 \left( \ln(1 + 1/\epsilon) \right)^d.
		\end{equation*}
		To conclude, note that $\epsilon \leq \epsilon_{\max} < \frac{1}{2}$ and thus $\ln (1 + 1/\epsilon) \leq 2 \lvert \ln \epsilon \rvert$ and
		$1 \leq \lvert \ln \epsilon \rvert \leq \lvert \ln \epsilon \rvert^d$.
		\end{proof}
Putting everything together and recalling $\delta = \left(4 \epsilon \right)^{1/\kappa}$
we find
\begin{align*}
\EE \left[ \mathrm{Tr} \left[ \chi_{[E- \epsilon, E + \epsilon]}(H_{\omega,L}) \right] \right]
& \leq
\left( K_1 \euler^{b} + 2^d K_2  \right)
\lVert\nu_\mu\rVert_\infty
\cdot
\delta
\left\lvert\ln \epsilon \right\rvert^d \lvert \tilde \Lambda_L \rvert \\
& \leq
\left( K_1 \euler^{b} + 2^d K_2  \right)
\lVert\nu_\mu\rVert_\infty
\cdot
\left(4 \epsilon \right)^{1/\kappa}
\left\lvert\ln \epsilon \right\rvert^d (2 / G_1)^d L^d. \qedhere
\end{align*}
\end{proof}
\begin{proof}[Proof of Theorem~\ref{thm:initial}]
We follow the ideas developed in \cite{BarbarouxCH-97b,KirschSS-98a}.
Let $t \leq \delta_{\mathrm{max}}$, $V_{t,L}$ be the restriction of $V_\omega$ to $\Lambda_L$ obtained by setting all random variables to $t$, and $H_{t,L} = -\Delta_{\Lambda_L} + V_{L,t}$ on $L^2 (\Lambda_L)$ with Dirichlet boundary conditions. Note that $H_{0,L} = -\Delta_{\Lambda_L} + V_{0,L}$ and that the first eigenvalue of $H_{t,L}$ is bounded from above by $d(\pi / L)^2 + K_u$. Ineq.~\eqref{eq:eigenvalue_lifting_kappa} with $b = d\pi^2 + K_u$, $\omega_k = 0$, $k \in \mathcal{D}$, and $\delta = t$ yields that there is $\kappa = \kappa (d,\delta_{\mathrm{max}} , \au , \ao , \bu , \bo , G_2 , G_u , K_u)$ such that for all $t \leq \delta_{\mathrm{max}}$
\[
 \lambda_1 (H_{t,L}) \geq \lambda_1 (H_{0,L}) + t^{\kappa} .
\]
We choose $t = L^{-7 / (4\kappa)}$ and $L$ sufficiently large such that $t < \min\{\delta_{\mathrm{max}} , t_0\}$. Then,
\[
 \lambda_1 (H_{t,L}) - \lambda_1 (H_{0,L}) \geq L^{-7/4}.
\]
Let $\Omega_0 := \{\omega \in \Omega : \lambda_1 (H_{\omega,L}) \geq \lambda_1 (H_{t,L})\}$.
Since the potential values in $\Lambda_L$ only depend on $\omega_k$, $k \in \Lambda_{L+G_u} \cap \mathcal{D}$, we calculate using $\lvert \Lambda_{L + G_u} \cap \mathcal{D} \rvert \leq \lceil(L+G_u)/G_1\rceil^d$ and our assumption on the measure $\mu$ that
\begin{equation*}
\PP (\Omega_0)
\geq 1- \PP (\exists \gamma \in \Lambda_{L+G_u} \cap \mathcal{D} \colon 		
     \omega_\gamma \leq t)
\geq 1 - \left\lceil\frac{L+G_u}{G_1}\right\rceil^d \mu ([0,t])
\geq 1 - \left\lceil\frac{L+G_u}{G_1}\right\rceil^d \frac{C}{L^{7d/4}} .
\end{equation*}
Since $\lceil (L+G_u)/G_1 \rceil^d \leq L^{5d/4}$ for $L$ sufficiently large, we obtain the statement of the theorem.
\end{proof}

\section{Proof of observability estimate}\label{s:proof-observability}
 We want to apply \cite[Theorem 2.2]{Miller-10} where we choose $A = \Delta_L - V_L$ on $L^2(\Lambda_L)$ with Dirichlet boundary conditions, $C = \chi_{W_\delta(L)}$ and $C_0 = \mathrm{Id}$.
 Note that $A$ is self-adjoint with spectrum contained in $(- \infty, \lVert V \rVert_\infty]$.
 For $\lambda > 0$ we define the increasing sequence of spectral subspaces $\mathcal{E}_\lambda := \mathrm{Ran} \chi_{[- \lambda, \infty)}(\Delta_L - V_L)$.
 \par
 We need to check \cite[(5),(6),(7)]{Miller-10}.
 By spectral calculus, we have for all $\lambda > 0$
 \begin{equation*}
  \lVert \mathrm{e}^{(\Delta_L - V_L) t} u \rVert_{\Lambda_L} \leq \mathrm{e}^{-\lambda t} \lVert u \rVert_{\Lambda_L},
  \quad u \in \mathcal{E}_\lambda^\perp = \mathrm{Ran} \chi_{(- \infty, - \lambda)}(\Delta_L - V_L),
  \quad t > 0.
 \end{equation*}
 Furthermore, Corollary \ref{cor:result1} implies for all $\lambda > 0$ and $u \in \mathcal{E}_{\lambda}$
\begin{equation*}
	\label{eq:LM6}
	\lVert u\rVert_{\Lambda_L}^2 \leq a_0 \mathrm{e}^{- N \ln (\delta/G) G \sqrt{\lambda}} \lVert u \rVert_{W_\delta(L)}^2.
\end{equation*}
For $T \leq 1$ we have
$\euler^{2 T \lVert V \rVert_\infty} / T \leq \euler^{2 \lVert V \rVert_\infty} \euler^{2/T}$
whence
\begin{equation*}
  \lVert \euler^{T(\Delta - V)} u \rVert_{\Lambda_L}^2
  \leq
  \frac{\mathrm{e}^{2 T \lVert V \rVert_\infty}}{T} \int_0^T \lVert \mathrm{e}^{t(\Delta - V)} u \rVert_{\Lambda_L}^2 \mathrm{d}t
  \leq
  \euler^{2 \lVert V \rVert_\infty} \mathrm{e}^{2/T} \int_0^T \lVert \mathrm{e}^{t(\Delta - V)} u \rVert_{\Lambda_L}^2 \mathrm{d}t .
 \end{equation*}
 Thus we found \cite[(5),(6),(7)]{Miller-10} with $m_0 = 1$, $m = 0$, $\alpha = \nu = 1/2$, $a_0$ and $b_0$ as in the theorem, $a = - (N/2) \ln (\delta / G) G > 0$, $b = 1$ and $\beta = 1$.
 By \cite[Theorem 2.2 and Corollary 1 (i)]{Miller-10}, there exists $T' > 0$ such that for all $T \leq T'$
 \[
  \kappa_T \leq  4 a_0 b_0 \euler^{2 c_\ast / T},\
\text{where}\
 c_\ast = 4 \left( \sqrt{ a +2} - \sqrt{a} \right)^{-4} .
\]
From the proof in \cite{Miller-10}, it can be inferred that $T'$ only depends on $m_0$, $\alpha$, $\beta$, $a$, $b$, $a_0$, $b_0$ and on our choice $T \leq 1$. Thus, in our case, $T'$ only depends on $G$, $\delta$ and $\lVert V \rVert_\infty$.
Using $\sqrt{a + 2} - \sqrt{a} = \int_{a}^{a+2} (2 \sqrt{x})^{-1} \mathrm{d}x \geq (a+2)^{-1/2}$ and the fact that from $\delta \leq G/2$, it follows that $2 \leq 2a/a_{\min}$ where $a_{\min} := (N/2) \ln (2) G$, and we obtain
\[
  c_{\ast} \leq 4 ( a + 2 )^2
  \leq 4 a^2 ( 1 + 2/a_{\min})^2
  =
  \ln (G/\delta)^2  \left( NG + 4 / \ln 2 \right)^2. \qedhere
\]
%
%%%%%%%%%%%%
%%%
%%% Appendix
%%%
%%%%%%%%%%%%
%
\appendix
\section{Sketch of proof of Proposition~\ref{prop:carleman2}} \label{app:carleman}
We follow \cite{BourgainK-05,NakicRT-15} and consider the case $\rho = 1$ and $u \in C_{\mathrm{c}}^\infty (B (1) \setminus \{0\} ; \RR)$ only. The general case follows by  regularization ($u \in W^{2,2} (\RR^d)$ with support in $B (1) \setminus \{0\}$), scaling (to $\rho > 0$), and adding the two Carleman estimates for the real and imaginary parts of $u$.
Let $\sigma: \RR^d \to \RR$ be given by $\sigma (x) = \lvert x \rvert$, $\phi (s) = \mathrm{e}^s$, $g = w^{-\alpha} u $,
$w(x)=\psi(\lvert x \rvert)$,
\[
\psi(s) = s \, \exp\left[-\int_0^s\frac{1-\mathrm{e}^{-t}}t\mathrm{d}t\right],
 \quad
 \tilde \nabla g
 = \nabla g - \frac{\nabla \sigma^\T \nabla g}{\lvert \nabla \sigma \rvert^2} \nabla \sigma
 = \nabla g - \frac{\nabla w^\T \nabla g}{\lvert \nabla w \rvert^2} \nabla w,
\]
$F_w := (w\Delta w - \lvert \nabla w \rvert^2) / \lvert \nabla w \rvert^2: \RR^d \to \RR$
and $A (g):= (w \nabla w^{\mathrm T} \nabla g) / \lvert \nabla w \rvert^2 + (1/2) g F_w : \RR^d \to \RR$.
We follow the proof of \cite[Lemma~3.15]{BourgainK-05} until the estimate (8.2) in \cite{BourgainK-05}, i.e.
\begin{equation}\label{eq:intermediate1}
 4 \alpha^2 \int \frac{\lvert \nabla w \rvert^2}{w^2} A(g)^2 + 2 \alpha \int \sigma \phi' (\sigma) \lvert \tilde\nabla g \rvert^2 + 2 \alpha^3 \int \sigma \phi' (\sigma) \frac{\lvert \nabla w \rvert^2}{w^2} g^2
 \leq \int \frac{w^2}{\lvert \nabla w \rvert^2} \bigl( w^{-\alpha} \Delta u \bigr)^2 +  R_1  ,
\end{equation}
where
\[
 R_1 = C \left( \alpha \int w^{1-\alpha} \lvert g \rvert \lvert \Delta u \rvert + \alpha \int w^{-1} g^2 + \alpha \int w \frac{\lvert \nabla \sigma^\T \nabla g \rvert^2}{\lvert \nabla \sigma \rvert^2} + \alpha^2 \int w^{-1} \lvert A(g) \rvert \lvert g \rvert \right) .
\]
As explained in \cite{BourgainK-05},
one can drop the positive term  $\int \sigma \phi'(\sigma) \lvert \tilde\nabla g \rvert^2$ in \eqref{eq:intermediate1},
and obtain for sufficiently large $\alpha$ the Carleman estimate
 \begin{equation} \label{eq:carleman1}
 \alpha^3 \int_{\RR^d}  w^{-1-2\alpha} u^2 \leq \tilde C_2 \int_{\RR^d} w^{2-2\alpha} \left( \Delta u \right)^2 .
\end{equation}
Following now \cite{NakicRT-15} we do not drop the term $\int \sigma \phi'(\sigma) \lvert \tilde\nabla g \rvert^2$ and use instead
\begin{equation} \label{eq:gradient}
\lvert \tilde \nabla g \rvert^2
 = w^{-2\alpha} \lvert \nabla u \rvert^2
 - 2 \alpha w^{-2}    g \lvert \nabla w\rvert^2  A (g) 
 +  \alpha w^{-2}    g^2 F_w \lvert \nabla w \rvert^2
 - \alpha^2 w^{-2} g^2 \lvert \nabla w \rvert^2
 - \frac{(\nabla w^\T \nabla g)^2}{\lvert \nabla w \rvert^2}  .
\end{equation}
Combining Eq.~\eqref{eq:gradient} with Ineq.~\eqref{eq:intermediate1},
and using the bounds $F_w \geq -C_F=\inf_{B_1^\circ} F_w$ and $w \leq \sigma \phi'(\sigma)$, we obtain
\begin{multline} \label{eq:intermediate2}
4\alpha^2\int\frac{\lvert \nabla w \rvert^2}{w^2}A (g)^2
+ 2 \alpha\int w^{-2\alpha + 1} \lvert \nabla u \rvert^2
- 2 C_F \alpha^2 \int\sigma \phi' (\sigma) \frac{\lvert \nabla w \rvert^2}{w^2} g^2
\\
\leq
\int \frac{w^2}{\lvert \nabla w\rvert^2}(w^{-\alpha} \Delta u)^2 + R_2,
\end{multline}
with some appropriate rest term $R_2$. If we compare Ineqs.~\eqref{eq:intermediate1} and \eqref{eq:intermediate2}, we observe that the required gradient term is now included, while the $g^2$-term, which corresponds to the lower bound of Ineq.~\eqref{eq:carleman1}, is now negative and goes with $\alpha^2$ instead of $\alpha^3$! In a similar way as Ineq.~\eqref{eq:intermediate1} implies Ineq.~\eqref{eq:carleman1}, one calculates that Ineq.~\eqref{eq:intermediate2} implies for sufficiently large $\alpha$
\begin{equation} \label{eq:carleman2}
  \alpha \int_{\RR^d} w^{1-2\alpha} \lvert \nabla u \rvert^2-\alpha^2 \int_{\RR^d}  w^{-1-2\alpha} u^2   \leq \hat C_2 \int_{\RR^d} w^{2-2\alpha} \left( \Delta u \right)^2 \drm x .
\end{equation}
By adding the two estimates \eqref{eq:carleman1} and \eqref{eq:carleman2} we obtain the desired estimate by choosing $\alpha$ sufficiently large.
\ifthenelse{\boolean{journal}}{}{
\section{Constants} \label{app:constants} 
 \subsection{Cutoff functions}
 Let $f,\psi : \RR \to [0,1]$ be given by
\[
 f (x) = \begin{cases}
          \mathrm{e}^{-1/x} & x > 0 , \\
          0 & x \leq 0 ,
         \end{cases}
\quad \text{and} \quad
\psi (x) = \frac{f(x)}{f(x)  + f(1-x)} .
\]
Note that the function $\psi$ is $C^\infty (\RR)$ and satisfies
\[
 \sup_{x \in \RR} \psi' (x) \leq 2 =:C', \quad
 \sup_{x \in \RR} \psi'' (x) \leq 10 =: C'' , \quad \text{and} \quad
 \psi(x) = \begin{cases}
            0 & \text{if $x \leq 0$}, \\
            1 & \text{if $x \geq 1$} .
           \end{cases}
\]
For $\epsilon > 0$ we define $\psi_\epsilon : \RR \to [0,1]$ by
\[
 \psi_\epsilon (x) = \psi (x / \epsilon) .
\]
Let now $M \subset \RR^{d+1}$ and $h_M : \RR^{d+1} \to \RR$ with $h_M (x) \geq \operatorname{dist} (x,M)$ if $x \not \in M$ and $h_M (x) \leq 0$ if $x \in M$. For $\epsilon > 0$ we define $\chi : \RR^{d+1} \to [0,1]$ by
\[
 \chi_{M,\epsilon} (x) = \psi_\epsilon \bigl(\epsilon - h_M (x) \bigr) .
\]
Of course, $h_M(x) := \mathrm{dist}(x,M)$ is a possible choice, but in applications we will require $h_M$ to have certain additional properties.
By construction we have (cf.\ Fig.~\ref{fig:cutoff})
\[
 \chi_{M,\epsilon} (x) = \begin{cases}
             1 & \text{if $x \in M$} , \\
             0 & \text{if $\operatorname{dist} (x,M) \geq \epsilon$} .
            \end{cases}
\]
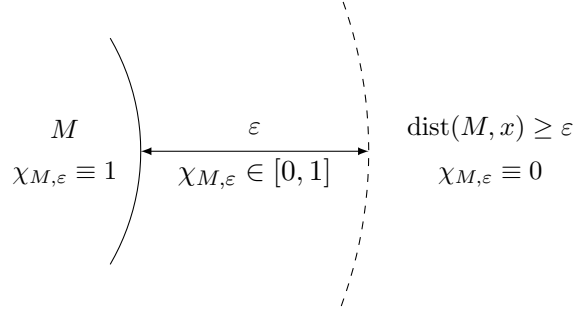
\begin{figure}[ht] \centering
\begin{tikzpicture}[xscale=1]
\draw [domain=-30:30] plot ({3*cos(\x)}, {3*sin(\x)});
\draw [dashed,domain=-20:20] plot ({6*cos(\x)}, {6*sin(\x)});
% \draw[dotted] (1,0) arc (-30:28:4.1);
 \draw (2,0.3) node {\small $M$};
 \draw (2,-0.3) node {\small $\chi_{M,\epsilon} \equiv 1$};
 \draw[latex-latex] (3,0) -- (6,0);
 \draw (4.5,0.3) node {\small $\epsilon$};
 \draw (4.5,-0.3) node {$\chi_{M,\epsilon} \in [0,1]$};
 \draw (7.6,0.3) node {\small $\operatorname{dist} (M,x) \geq \epsilon$};
 \draw (7.6,-0.3) node {\small $\chi_{M,\epsilon} \equiv 0$};
\end{tikzpicture}
\caption{Cutoff function $\chi_{M,\epsilon}$\label{fig:cutoff}}
\end{figure}
\subsubsection{The constants $\cut_2$ and $\cut_3$}
We want to construct a cutoff function $\chi \in C_{\mathrm c}^\infty (\RR^{d+1};[0,1])$ with $\supp \chi \subset B(\rout_3) \setminus \{0\}$ and $\chi (x) = 1$ if $x \in B(\rin_3) \setminus \overline{B(\rout_1)}$. We set $\tilde M = B (r_3)$, $2 \tilde \epsilon = R_3 - r_3$, $h_{\tilde M} (x) = \vert x \rvert - r_3$ and define
\[
 \tilde \chi (x) = \chi_{\tilde M, \tilde\epsilon} (x) .
\]
Note that
\[
 \tilde \chi (x) = \begin{cases}
                    1 & \text{if $x \in B (r_3)$}, \\
                    0 & \text{if $x \not\in B ((r_3 + R_3)/2)$}.
                   \end{cases}
\]
For the partial derivatives we calculate
\begin{align*}
 (\partial_i \tilde \chi) (x) &= -\frac{1}{\tilde\epsilon} \psi' (1 - h_{\tilde M} (x) / \tilde\epsilon) \frac{x_i}{\lvert x \rvert} , \\
 (\partial^2_i \tilde \chi) (x) &= \frac{1}{\tilde\epsilon^2} \psi'' (1 - h_{\tilde M} (x) / \tilde\epsilon) \frac{x_i^2}{\lvert x \rvert^2} - \frac{1}{\tilde\epsilon} \psi' (1 - h_{\tilde M} (x) / \tilde\epsilon) \left( \frac{1}{\lvert x \rvert} - \frac{x_i^2}{\lvert x \rvert^3} \right) .
\end{align*}
Hence, using $\Delta\tilde \chi (x) = 0$ if $x \not \in B (R_3) \setminus B (r_3)$ and $2 \tilde\epsilon = R_3 - r_3=3\euler\sqrt{d}$, we obtain
\begin{align*}
 \lVert \nabla \tilde\chi \rVert_\infty &\leq \frac{C'}{\tilde\epsilon} = \frac{4}{R_3-r_3} = \frac{4}{3\euler\sqrt{d}} \leq 1 ,\\
 \lVert \Delta \tilde\chi \rVert_\infty & \leq \frac{C''}{\tilde\epsilon^2} + \frac{C'}{\tilde\epsilon} \frac{d}{r_3}
 \leq \frac{80+4d}{18\euler^2 d} \leq \frac{84}{18\euler^2} \leq 1 .
\end{align*}
Analogously we find a function $\hat \chi$ with values in $[0,1]$, $\hat \chi (x) = 0$ if $x \in B (r_1)$, $\hat \chi (x) = 1$ if $x \not \in B (R_1)$ and, using $R_1 - r_1 = r_1 \geq \delta^2 / 64$,
\begin{align*}
 \lVert \nabla \hat \chi \rVert_\infty
  &\leq \frac{C'}{R_1-r_1} \leq \frac{128}{\delta^2} ,\\
 \lVert \Delta \tilde\chi \rVert_\infty & \leq \frac{C''}{(R_1 - r_1)^2} + \frac{C'}{(R_1 - r_1)} \frac{d}{r_1}
 \leq \frac{10\cdot64^2}{\delta^4} + \frac{2 d 64^2}{\delta^4} \leq \frac{12d64^2}{\delta^4} .
\end{align*}
Our cutoff function $\chi \in C_{\mathrm c}^\infty (\RR^{d+1};[0,1])$ with $\supp \chi \subset B(\rout_3) \setminus \{0\}$ and $\chi (x) = 1$ if $x \in B(\rin_3) \setminus \overline{B(\rout_1)}$ can be defined by
\[
 \chi (x) = \begin{cases}
             \chi (x) = \hat\chi (x) & \text{if $x \in B (R_1) \setminus \overline{B (r_1)}$}, \\
             \chi (x) = 1 & \text{if $x \in B (r_3) \setminus \overline{B (R_1)}$}, \\
             \chi (x) = \tilde \chi (x) & \text{if $x \in B (R_3) \setminus \overline{B (r_3)}$},
            \end{cases}
\]
and has the properties (recall $V_i = B (R_i) \setminus \overline{B (r_i)}$)
\[
 \max\{\lVert \Delta \chi \rVert_{\infty , V_1} , \lVert \lvert \nabla \chi \rvert \rVert_{\infty , V_1}\} \leq \frac{12d64^2}{\delta^4} =: \frac{\tilde\cut_2}{\delta^4} =: \cut_2
\]
and
\[
 \max\{\lVert \Delta \chi \rVert_{\infty , V_3} , \lVert \lvert \nabla \chi \rvert \rVert_{\infty , V_3}\} \leq \frac{4}{3\euler} =: \cut_3 .
\]

\subsubsection{The constant $\cut_1$}
We choose $M = S_2$, $\epsilon = \delta^2 / 16$ and
\[
 h_{S_2} (x) = x_{d+1} - 1 + \sqrt{ a_2^2 + \frac{\lvert x' \rvert^2}{2 } } .
\]
Obviously, $h_{S_2} (x) \geq \operatorname{dist} (x,S_2)$ if $x \not \in S$ and $h_{S_2} (x) \leq 0$ if $x \in S_2$, cf.\ Fig.~\ref{fig:dist_hyp}.
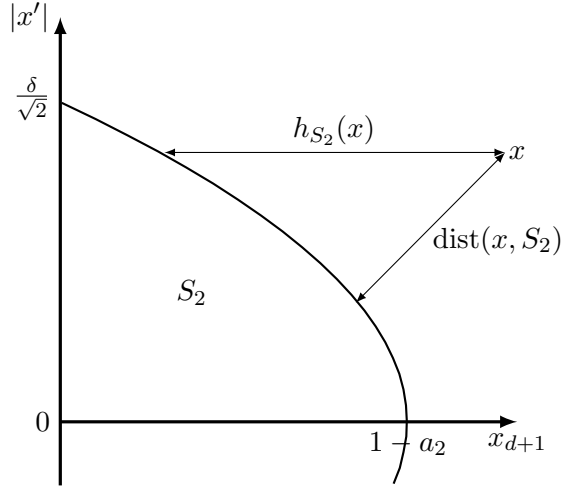
\begin{figure}[ht] \centering
\begin{tikzpicture}[scale=12,xscale=12]
%  \clip (-0.1,0.1) rectangle (-0.1,0.1);
%
\pgfmathsetmacro{\d}{0.5}
%
% h1
%\a = a_3
%\b = b_3
\pgfmathsetmacro{\a}{sqrt(1-\d*\d/2)}
\pgfmathsetmacro{\b}{sqrt(2)*\a}
\pgfmathsetmacro{\aa}{sqrt(1-\d*\d/4)}
\pgfmathsetmacro{\bb}{sqrt(2)*\aa}

\draw[thick] plot[domain=-0.05:0.255]({-\aa*cosh(\x)},{\bb*sinh(\x)});
\draw[-latex, very thick] (-1,-2pt)--(-1,\d/1.44 +0.1) node[left] {$\lvert x' \rvert$};
\draw[-latex, very thick] (-1,0)--(1-\aa-1+0.01,0) node[below] {$x_{d+1}$};
\draw[] (-1,\d/1.4) node[left] {$\frac{\delta}{\sqrt{2}}$};
\draw[] (-\aa,0) node[below] {$1 - a_2$};
\draw[] (-1cm,0) node[left] {$0$};
\draw (1-\aa-1+0.01 , \d/1.68) node {$x$};
\draw[latex-latex] (1-\aa-1+0.0089 , \d/1.68) -- (-0.9905 , \d/1.68);
\draw[latex-latex] (1-\aa-1+0.0089 , \d/1.68) -- (-0.973 , \d/3.8);
\draw (-0.975 , \d/1.55) node {$h_{S_2} (x)$};
\draw (-0.96 , \d/2.5) node {$\operatorname{dist} (x,S_2)$};
\draw (-0.988 , \d/3.5) node {$S_2$};
\end{tikzpicture}
\caption{Illustration of the hyperbolas $h_2$ and $h_3$\label{fig:dist_hyp}}
\end{figure}
Since the distance between the sets $S_2$ and $\RR^{d+1}_+ \setminus S_3$ is bounded from below by $\delta^2 / 16$, see Appendix~\ref{app:distance}, we find that
\[
 \chi_{S,\epsilon} (x) = \begin{cases}
             1 & \text{if $x \in S_2$} , \\
             0 & \text{if $x \in \RR^{d+1}_+ \setminus S_3$} .
            \end{cases}
\]
For the partial derivatives we calculate for $x \in S_3 \setminus S_2$
\[
 (\partial_i \chi)(x) = -\frac{1}{\epsilon} \psi' (1-h_{S_2} (x) / \epsilon)
 \begin{cases}
\frac{x_i}{2} \left(a_2^2 + \frac{\lvert x' \rvert^2}{2 }\right)^{-1/2} & \text{if $i \in \{1,\ldots , d\}$}                                                                       , \\
1 & \text{if $i = d+1$} ,
\end{cases}
\]
and find by using $\lvert x' \rvert^2 \leq 1/4$ for $x \in S_3 \setminus S_2$ and $a_2^2 \in [15/16 , 1]$
\[
 \lVert \nabla \chi_{S , \epsilon} \rVert_\infty^2 \leq \frac{16}{466}\left(\frac{C'}{\epsilon} \right)^2 , \quad \text{hence,} \quad
 \lVert \nabla \chi_{S , \epsilon} \rVert_\infty \leq \frac{6}{\delta^2} .
\]
For the second partial derivatives we calculate for $i \in \{1,\ldots , d\}$
\begin{multline*}
 (\partial_i^2 \chi)(x) =
 \frac{1}{\epsilon^2} \psi'' (1-h_S (x)/\epsilon) \frac{x_i^2}{4} \left(a_2^2 + \frac{\lvert x' \rvert^2}{2} \right)^{-1}  \\
 - \frac{1}{\epsilon} \psi' (1-h_S (x)/\epsilon) \left[ \frac{1}{2} \left( a_2^2 + \frac{\lvert x' \rvert^2}{2} \right)^{-1/2} - \frac{x_i^2}{4} \left( a_2^2 + \frac{\lvert x' \rvert^2}{2} \right)^{-3/2} \right] ,
\end{multline*}
and $\partial_{d+1}^2 \chi (x) = (1/\epsilon^2) \psi'' (1-h_S (x)/\epsilon)$.
Hence, using $\lvert x' \rvert^2 \leq 1/4$ for $x \in S_3 \setminus S_2$ and $a_2^2 \in [15/16 , 1]$
\[
 \lVert \Delta \chi \rVert_\infty  \leq
  \frac{C''}{\epsilon^2} \frac{237}{233}
   + \frac{C'}{2\epsilon a_2} (d+8/233) \leq \frac{16^2 \cdot 11d}{\delta^4} =: \frac{\tilde\cut_1}{\delta^4} =: \cut_1.
\]
\subsubsection{Distance of $S_2$ and $\RR^{d+1}_+ \setminus S_3$}\label{app:distance}
The distance between the sets $S_2$ and $\RR^{d+1}_+ \setminus S_3$ is given by the distance between the two hyperbolas
\[
 h_i \colon \frac{(x-1)^2}{a_i^2} - \frac{y^2}{b_i^2} = 1 , \quad i \in \{2,3\}
\]
in $\{(x,y) \in \RR^2 \colon x,y \geq 0\}$, where $a_i$ and $b_i$ are given by
\[
 a_2^2 = 1 - \frac{\delta^2}{4}, \quad
 a_3^2 = 1 - \frac{\delta^2}{2} \quad \text{and} \quad
 b_i^2 = 2 a_i^2 .
\]
See Fig.~\ref{fig:hyperbolas} for an illustration.
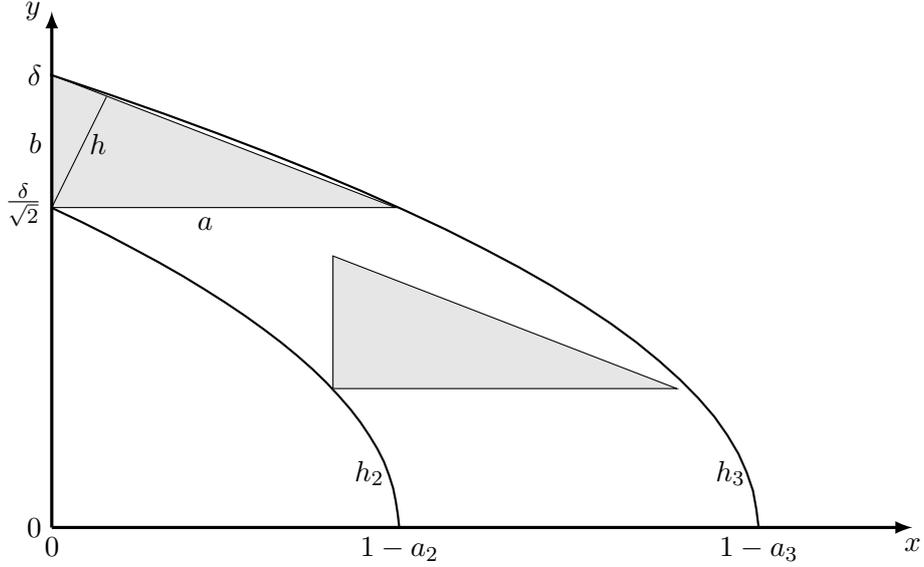
\begin{figure}[ht]\centering
\begin{tikzpicture}[scale=12,xscale=12]
\pgfmathsetmacro{\d}{0.5}
%
% h1
%\a = a_3
%\b = b_3
\pgfmathsetmacro{\a}{sqrt(1-\d*\d/2)}
\pgfmathsetmacro{\b}{sqrt(2)*\a}
%Dreieck1
\filldraw[black!10] (-1, \d/1.414) -- (-0.9686,\d/1.414) -- (-1,\d)--(-1, \d/1.414);
\draw[] (-1, \d/1.414) -- (-0.9686,\d/1.414) -- (-1,\d)--(-1, \d/1.414);
\draw[thick] plot[domain=-0:0.37]({-\a*cosh(\x)},{\b*sinh(\x)});
% h2
%\aa = a_2
%\bb = b_2
\pgfmathsetmacro{\aa}{sqrt(1-\d*\d/4)}
\pgfmathsetmacro{\bb}{sqrt(2)*\aa}
\draw[thick] plot[domain=-0:0.255]({-\aa*cosh(\x)},{\bb*sinh(\x)});
\draw[-latex, very thick] (-1,0)--(-1,\d cm+2pt) node[left] {$y$};
\draw[-latex, very thick] (-1,0)--(-\a cm + 0.4pt,0) node[below] {$x$};
\draw[] (-1,\d) node[left] {$\delta$};
\draw[] (-1,\d/1.4) node[left] {$\frac{\delta}{\sqrt{2}}$};
\draw[] (-\a,0) node[below] {$1 - a_3$};
\draw[] (-\aa,0) node[below] {$1 - a_2$};
\draw[] (-1cm,0) node[below] {$0$};
\draw[] (-1cm,0) node[left] {$0$};
\draw[] (-0.938,0.06) node {$h_3$};
\draw[] (-0.971,0.06) node {$h_2$};
%
%Dreieck2
\draw[pattern=north west lines, very thin, black!10] (-1+0.0257, \d/1.414-0.2) -- (-0.9686+0.0257,\d/1.414-0.2) -- (-1+0.0257,\d-0.2) -- (-1+0.0257, \d/1.414-0.2);
\draw[] (-1+0.0257, \d/1.414-0.2) -- (-0.9686+0.0257,\d/1.414-0.2) -- (-1+0.0257,\d-0.2) -- (-1+0.0257, \d/1.414-0.2);
\draw[] (-0.986,\d/1.414) node[below] {$a$};
\draw[] (-1,0.425) node[left] {$b$};
\draw[] (-0.9941,\d/1.414+0.07) node[left] {$h$};
%
%Hoehe
\draw (-1,\d/1.414)--(-0.995, \d-0.024);
\end{tikzpicture}
\caption{Illustration of the hyperbolas $h_2$ and $h_3$}
\label{fig:hyperbolas}
\end{figure}
By symmetry we can consider the case $y \geq 0$ only.
First we show that in order to estimate the distance between $h_2$ and $h_3$ from below, it is sufficient to consider the distance between the intersection point of $h_2$ with the $x$-axis and $h_3$.
For every point $(x,y)$ on $h_2$, we define the distance $a(y)$ between $h_2$ and $h_3$ in $x$-direction and the distance $b(x)$ in $y$-direction.
This gives rise to a rectangular triangle with catheti of length $a$ and $b$.
Due to concavity and monotonicity of $h_2$ and $h_3$, considered as functions of $x$, a lower bound for the distance of $(x,y)$ to $h_3$ is given by the height of this rectangular triangle, given by
\[
h (x) := \frac{a(x) b (x)}{\sqrt{a^2 (x)+b^2 (x)}}.
\]
By a straightforward calculation, we see that $b(x)$ is strictly increasing as a function of $x$ while $a(y)$ is strictly decreasing as a function of $y$.
Thus, taking the triangle at the point $(0,\delta/ \sqrt{2})$ and moving it along $h_2$, the triangle will always stay below $h_3$, see Fig.~\ref{fig:hyperbolas}. Hence, $h$ evaluated at the point $(0,\delta/\sqrt{2})$ is a lower bound for $\mathrm{dist} (h_2,h_3)$.
We have
\begin{equation*}
a(\delta/\sqrt{2}) = 1 - \sqrt{1 - \frac{\delta^2}{4}}  \quad \text{and} \quad
b(0)  = \left( 1 - \frac{1}{\sqrt{2}} \right) \delta.
\end{equation*}
Hence,
\begin{equation*}
\mathrm{dist}(h_2,h_3) \geq \frac{ \left( 1 - \frac{1}{\sqrt{2}} \right) \delta \left( 1 - \sqrt{1 - \frac{\delta^2}{4}} \right) }{ \sqrt{ \left( 1 - \frac{1}{\sqrt{2}} \right)^2 \delta^2 + \left( 1 - \sqrt{1 - \frac{\delta^2}{4}} \right)^2 }}.
\end{equation*}
We use $\delta^2 / 8 \leq  1- \sqrt{1 - \delta^2 / 4} \leq \delta/2$ and obtain the bound
\[
 \mathrm{dist}(h_2,h_3) \geq \frac{\left( 1 - \frac{1}{\sqrt{2}} \right) \delta^2}{8 \sqrt{\left( 1 - \frac{1}{\sqrt{2}} \right)^2 + 1/4}}  > \frac{\delta^2}{16}  .
\]
\subsection{The constant $\protect{\tilde C_{\sfuc}}$} \label{app:final_const}
We estimate $\tilde C_{\sfuc} = D_1^{-4} \left( D_2 D_3 D_4 \right)^{-4/\gamma}$.
We start by estimating the constants $D_i$, $i \in \{1,\ldots,  4\}$ separately. By $K_i$, $i \in \{1,\ldots , 11\}$ we will denote positive constants which do not depend on $\delta$, $b$ and $\lVert V \rVert_\infty$, and will change from line to line. We will frequently use $\delta^2 / 64 \leq \rin_1 \leq \delta/8$, $K_1 \geq (1/2)^{K_2} \geq \delta^{K_2}$, and $a^{\ln b} = b^{\ln a}$ for $a,b > 0$.
For $D_1$ and $D_2$ we calculate
\[
  D_1^{-4} \geq \delta^{K_1 \bigl(1+\lVert V \rVert_\infty^{2/3} \bigr)} ,
  \quad \text{and} \quad
  D_2^{-4/\gamma} \geq \delta^{K_2 \bigl(1+ \lVert V \rVert_\infty^{2/3} \bigr)} .
\]
For the constant $D_3$ we have
\begin{equation*}
 D_3^{-4/\gamma}
 = \left( \mathrm{e}^{4 R_3 \sqrt{b}} \right)^{-\ln (r_3 / r_1) / \ln 2}
 = \left( \frac{r_1}{r_3} \right)^{\ln (\mathrm{e}^{4 R_3 \sqrt{b}}) / \ln 2} \geq \delta^{K_1 (1+ \sqrt{b})} .
\end{equation*}
For the constant $D_4$ we have $D_4^2 \leq K_1 (1+ \lVert V \rVert_\infty)$ and hence
\begin{align*}
D_4^{-4/\gamma} & \geq K_1^{-2/\gamma} (1+\lVert V \rVert_\infty)^{-2/\gamma}
\geq \delta^{K_2} \delta^{K_3 \ln (1+\lVert V \rVert_\infty)} \geq \delta^{K_4 (1+\lVert V \rVert_\infty^{2/3})} .
\end{align*}
Hence, we obtain the desired behaviour
\[
\tilde C_{\sfuc} \geq \delta^{K_1 (1+ \lVert V \rVert_\infty^{2/3} + \sqrt{b})} .
\]
}
%
%%%%%%%%%%%%
%%%
%%% Comments on breather potentials
%%%
%%%%%%%%%%%%
%
\section{On single-site potentials for the breather model}
 \label{app:assumption_single_site}
\subsection{Our assumptions}
In this section we discuss our conditions on the single-site potential in the random breather model. Recall that the $\omega_j$ were supported in $[\omega_{-}, \omega_{+}] \subset [0,1)$ whence we consider $t \in [\omega_{-}, \omega_{+}]$ and $\delta \in [0, 1 - \omega_{+}]$.
 \begin{definition}
 We say that a family $\{ u_t \}_{t \in [0,1]}$ of measurable functions $u_t : \RR^d \to \RR$ satisfies condition
 \begin{itemize}
  \item[(A)] if the $u_t$ are uniformly bounded, have uniform compact support and if there are $\au,\bu > 0$ and $\ao, \bo \geq 0$ such that for all $t \in [\omega_{-}, \omega_{+}]$, $\delta  \leq 1 - \omega_{+}$ there is $x_0 = x_0(t,\delta) \in \RR^d$ with
\begin{equation}\label{eq:condition(A)}
 u_{t+\delta} - u_{t} \geq \au \delta^{\ao} \chi_{ B(x_0, \bu \delta^{\bo}) }.
\end{equation}
  \item[(B)] if $u_t$ is the dilation of a function $u$ by $t$, defined as $u_t (x) := u(x/t)$ for $t > 0$ and $u_0 \equiv 0$, where $u$ is the characteristic function of a bounded convex set $K$ with $0 \in \overline{K}$.
  \item[(C)] if $u_t$ is the dilation of a measurable function $u$ which is positive, radially symmetric, compactly supported, bounded with decreasing radial part $r_u:[0,\infty) \to [0, \infty)$ and such there is a point $\tilde x > 0$ where $r_u$ is differentiable, $r_u'(\tilde x) < 0$ and $r_u(\tilde x) > 0$.
  \item[(D)] if $u_t$ is the dilation of a measurable function $u$ which is positive, radially symmetric, radially decreasing, compactly supported, bounded and which has a discontinuity away from $0$.
  \item[(E)] if $u_t$ is the dilation of a measurable function which is non-positive, radially symmetric, radially increasing, compactly supported, bounded, and such there is a point $\tilde x > 0$ where the radial part $r_u$ is differentiable, $r_u'(\tilde x) > 0$ and $r_u(\tilde x) < 0$ .
 \end{itemize}
\end{definition}
\begin{remark}
\label{remark:(A)_to_(E)}
Condition (A) is the abstract assumption we used in the proof of the Wegner estimate for the random breather model. Conditions (B) to (E) are relatively easy to verify for specific examples of single-site potentials. In particular, (C) holds for many natural choices of single-site potentials such as the smooth function $\chi_{\lvert x \rvert < 1} \exp \left( 1 /( \lvert x \rvert^2 - 1) \right)$ or the hat-potential $\chi_{\lvert x \rvert < 1} (1 - \lvert x \rvert )$.
Furthermore, we note that if we have families $\{u_t\}_{t \in [0,1]}$ and $\{v_t\}_{t \in [0,1]}$ where $u_t$ satisfies (A) and $v_{t+\delta} - v_t \geq 0$ for all $t \in [\omega_{-}, \omega_{+}]$ and $\delta \in (0, 1 - \omega_{+}]$, then the family $\{u_t + v_t\}_{t \in [0,1]}$ also satisfies (A).
\end{remark}
\begin{lemma}
We have that each of the assumptions (B) to (E) implies (A).
\end{lemma}
\begin{proof}
Assume (B).
We will show (A) with $\au = 1$, $\ao = 0$, $\bo = 1$ and $\bu = c$, and hence it is enough to show the existence of a $c \delta$-ball in $K_{t + \delta} \backslash K_t$.

For $K \subset \RR^d$ and $t > 0$ we define $K_t := \{ x \in \RR^d : x/t \in K \}$ and $K_0 := \emptyset$.
 Without loss of generality let $x := (1,0,...,0)$ be a point in $\overline K$ which maximizes $\lvert x \rvert$ over $\overline K$.
 For $\lambda \in \RR$ define the half-space $H_\lambda := \{ x \in \RR^d : x_1 \leq \lambda \}$, where $x_1$ stands for the first coordinate of $x$.
 By scaling, the existence of a $c \delta$-ball in $K_{t + \delta} \backslash K_t$ is equivalent to the existence of a $c \delta / (t + \delta)$-ball in $K \backslash K_{t / (t + \delta)}$.
 By maximality of $(1,0,...,0)$, we have $K \subset H_1$ and hence $K_{t / (t + \delta)} \subset H_{t / (t + \delta)}$. Thus, it is sufficient to find a $c \frac{\delta}{t + \delta}$-ball in $K \backslash H_{t /(t + \delta)}$.
 By convexity of $K$, the set $\{ z \in K : z_1 = 1/2 \}$ is nonempty and since $K$ is open, we find $z_0 \in K$ with $z_1 = 1/2$ and $0 < c < 1/2$ such that $B(z_0,c) \subset K$.
We define for $\lambda \in [0,1)$ the set $X(\lambda) \subset \RR^d$ as $X(\lambda) := B(z_0 + \lambda ((1,0,...,0) - z_0), c \cdot (1 - \lambda))$.
By convexity and the fact that $(1,0,...,0) \in \overline K$, we have $X(\lambda) \subset K$.
In fact, let $\{x_n\}_{n \in \NN} \subset K$ be a sequence with $x_n \to (1,0,...,0)$. We define open sets $X_n(\lambda)$ by replacing $(1,0,...,0)$ by $x_n$ in the definition of $X(\lambda)$. By convexity of $K$, every $X_n$ is a subset of $K$ whence $\bigcup_{n \in \NN} X_n(\lambda) \subset K$. Furthermore we have $X(\lambda) \subset \bigcup_{n \in \NN} X_n(\lambda)$. Thus $X(\lambda) \subset K$.
We now choose $\lambda := \frac{t}{t + \delta}$.
Then $X(\lambda) \cap H_{\lambda} = \emptyset$.
Noting that $c ( 1 - \lambda ) = c \frac{\delta}{t + \delta}$, we see that $X(\lambda)$ is the desired $c \frac{\delta}{t + \delta}$-ball.
\par
Now we assume (C).
Let  $r_u'(\tilde x) = - C_1$. Then there is $\tilde \epsilon > 0$ such that
\begin{equation}
\label{eq:A_to_C_derivative}
 r_u(\tilde x + \epsilon) - r_u(\tilde x) \in \left[- 2 \epsilon C_1, \frac{- \epsilon}{2} C_1 \right]\quad \text{for all}\ \lvert \epsilon \rvert < \tilde \epsilon.
\end{equation}
It is sufficient to prove the following: There are $C_2, C_3 > 0$ such that for every $0 \leq t \leq \omega_{+}$ and every $0 < \delta \leq 1 - \omega_{+}$ there is $\hat x = \hat x(t, \delta)$ such that
\begin{equation}
\label{eq:A_to_C}
 r_u \left( \frac{\hat x + C_2 \delta}{t + \delta} \right) - r_u \left( \frac{\hat x}{t} \right) \geq C_3 \delta.
\end{equation}
Indeed, by monotonicity of $r_u$, \eqref{eq:A_to_C} implies that for every $x \in [\hat x, \hat x + C_2 \delta]$ we have
\[
 r_u \left( \frac{x}{t + \delta} \right) - r_u \left( \frac{x}{t} \right)
 \geq
 r_u \left( \frac{\hat x + C_2 \delta}{t + \delta} \right) - r_u \left( \frac{\hat x}{t} \right) \geq C_3 \delta
\]
whence (A) holds with $x_0 := (\hat x + C_2 \delta /2) e_1$, $\au = C_3$, $\bu = C_2/2$, $\ao = \bo = 1$.

%(A) => (C)

In order to see \eqref{eq:A_to_C}, let $\hat x =(t+\delta)\tilde{x}$.
We choose $\kappa \in (0, 1/4)$ and assume that $\tilde x - 4 \kappa \tilde \epsilon > 0$ (this is no restriction since \eqref{eq:A_to_C_derivative} also holds for smaller $\tilde \epsilon$).
Furthermore, we define $C_2 := \kappa \tilde \epsilon$.
Now we distinguish two cases. If $\tilde x \delta /t \leq \tilde \epsilon$, then \eqref{eq:A_to_C_derivative} implies
\begin{align*}
 r_u \left( \frac{\hat x + C_2 \delta}{t + \delta} \right) - r_u \left( \frac{\hat x}{t} \right)
 &=
 r_u \left( \tilde x + \kappa \frac{\tilde \epsilon \delta}{t + \delta} \right) - r_u \left( \tilde x \right)
 + r_u \left( \tilde x \right) - r_u \left( \tilde x + \tilde x \frac{\delta}{t}  \right)\\
 &\geq
 - 2 \kappa C_1 \frac{\tilde \epsilon \delta}{t + \delta} + C_1 \frac{\tilde x \delta}{2 t}  \ge \delta\frac{ C_1}{2} \frac{\tilde x - 4 \kappa \tilde \epsilon }{t+\delta} .
\end{align*}
If $\tilde x \delta /t > \tilde \epsilon$, we use $r_u(\tilde x) - r_u(\tilde x + \tilde x \delta / t) \geq r_u(\tilde x) - r_u(\tilde x + \tilde \epsilon)$ and \eqref{eq:A_to_C_derivative} to obtain
\[
  r_u \left( \frac{\hat x + C_2 \delta}{t + \delta} \right) - r_u \left( \frac{\hat x}{t} \right) \ge
   - 2 \kappa C_1 \frac{\tilde \epsilon \delta}{t + \delta} +C_1 \frac{\tilde \epsilon}{2} = \frac{C_1 \tilde \epsilon}{2} \left( 1- \frac{4 \kappa \delta}{t+\delta}   \right)  \ge \frac{C_1 \tilde \epsilon}{2} \left( 1- 4 \kappa \right).
\]
Hence
\begin{equation*}
  r_u \left( \frac{\hat x + C_2 \delta}{t + \delta} \right) - r_u \left( \frac{\hat x}{t} \right)  \geq C_3 \delta,
  \text{ where } C_3
 :=
 \min
 \left\{
 \frac{C_1 ( \tilde x - 4 \kappa \tilde \epsilon)}{2}, \frac{C_1 \tilde \epsilon (1 - 4 \kappa)}{2(1 - \omega_{+})}
 \right\} > 0 .
\end{equation*}
\par
The fact that (D) implies (A) is a consequence of (B).
In fact, a functions $u$ as in (D) can be decomposed $u = v + w$ where $v$ is (a multiple of) a characteristic function of a ball, centered at the origin, and $w$ is positive, radially symmetric and decreasing. Indeed, let $x_0$ be the point of discontinuity with the smallest norm. Then we can take $v = (u(x_0-)-u(x_0+))\chi_{B(0,\lvert x_0 \rvert )}$, where $\chi_A$ denotes the characteristic function of the set $A$.

The function $v$ satisfies (A) by (B) (since balls are convex) and we have $w_{t + \delta} - w_{t} \geq 0$. By Remark~\ref{remark:(A)_to_(E)}, the family $\{ u_t \}_{t \in [0,1]} = \{ v_t + w_t \}_{t \in [0,1]}$ also satisfies (A).
The case (E) is an adaptation of (C).
\end{proof}

\subsection{Earlier assumptions}
For certain types of random breather potentials Wegner estimates have been
given before, cf.\ \cite{CombesHM-96} and \cite{CombesHN-01}.
As we will show below, none of these results covers the \emph{standard breather model}.
The methods of \cite{CombesHM-96,CombesHN-01} seem to be motivated by reducing, thanks to linearization,
the random breather model to a model of alloy type and then applying methods designed for the latter one.
They are not focused to take advantage of the inherent, albeit non-linear, monotonicity of the random breather model.
The following assumptions on the single site potential are considered in
\cite{CombesHM-96} and \cite{CombesHN-01}, respectively.
\begin{definition}
 We say that a measurable function $u\colon \RR^d \to [0, \infty)$ satisfies condition
 \begin{itemize}
  \item[(F)] if $u$ is compactly supported, in $C^2(\RR^d)$, nonzero in a  neighbourhood of the origin and for some $c_0 > 0$ we have the inequalities
  \begin{equation}\label{Condition2_CHM}
   - x \cdot \nabla u \geq 0\ \text{for all}\ x \in \RR^d \quad
   \text{and} \quad
  \left\lvert \frac{ (x, \mathrm{Hess}[u]x)}{ x \cdot \nabla u} \right\rvert \leq c_0 < \infty\ \text{for all}\ x \in \RR^d \backslash \{ 0 \}.
  \end{equation}
  \item[(G)] if $u \not \equiv 0$ is compactly supported, in $C^1(B_1 \backslash \{0\})$, and there is $\epsilon_0 > 0$ such that
  \begin{equation}\label{Condition_CHN}
   - x \cdot \nabla u - \epsilon_0 u \geq 0 \ \text{for all}\ x \in \RR^d \backslash \{ 0 \}.
  \end{equation}
 \end{itemize}
\end{definition}
We have the following Lemma.
\begin{lemma}
 We have that
 \begin{itemize}
  \item (F) never holds,
  \item (G) implies that $u$ has a singularity at the origin.
 \end{itemize}
\end{lemma}
\begin{proof}
We first show the statements in dimension $1$.
Assume (F) and let $x_0 := \min \mathrm{supp}\ u$. Note that $x_0 < 0$.
By the first inequality in \eqref{Condition2_CHM} we have that $u' \geq 0$ for $x \in (x_0,0)$.
The second inequality in \eqref{Condition2_CHM} implies
\[
\lvert u''(x) \rvert \leq  \frac{c_0 u'(x)}{\lvert x \rvert} \leq \frac{2 c_0 u'(x)}{\lvert x_0 \rvert} \text{ for all}\ x \in (x_0, x_0 / 2)
\]
whence we have
\[
 u'(x) = \int_{x_0}^x u''(y) \mathrm{d} y \leq \int_{x_0}^x \lvert u''(y) \rvert \mathrm{d} y \leq \frac{2 c_0}{\lvert x_0 \rvert} \int_{x_0}^x u'(y) \mathrm{d}y
\]
and iteratively
\begin{align*}
 u'(x) \leq &
 \frac{(2 c_0)^n}{\lvert x_0 \rvert^n} \int_{x_0}^x \int_{x_0}^{x^{(1)}} ... \int_{x_0}^{x^{(n-1)}} u'(x^{(n)})\ \mathrm{d} x^{(n)} ... \mathrm{d} x^{(1)} \\
 \leq & \lVert u' \rVert_\infty \cdot  \frac{(2 c_0)^n}{\lvert x_0 \rvert^n}
\int_{x_0}^x \int_{x_0}^{x^{(1)}} ... \int_{x_0}^{x^{(n-1)}} \mathrm{d} x^{(n)} ... \mathrm{d} x^{(1)} \\
= &\lVert u' \rVert_\infty \cdot \left( \frac{ 2 c_0 (x - x_0)}{\lvert x_0 \rvert} \right)^n / n!\quad \rightarrow 0\ \text{as}\ n \to \infty
\end{align*}
for all $x \in x_0, x_0/2$. We found $u' \equiv 0$ on $(x_0,x_0/2)$,  which is a contradiction.
\par
\medskip
Now we assume (G). The function $u$ cannot have its supremum at a point of differentiability for else it would have to be zero at its maximum which would imply $u \equiv 0$.
Condition \eqref{Condition_CHN} implies that $u$ is increasing on the negative half axis and decreasing on the positive half axis.
We conclude that the supremum has to be the limit at the only possible non-differentiable point $x = 0$ and we will show that this limit is $\infty$.
By monotonicity of $u$ and the assumption $u\not\equiv 0$, there is $\delta_0 > 0$ such that
\begin{equation*}
u(x) \geq u(\delta_0) > 0\ \mathrm{ on }\ (0, \delta_0)\ \text{or}\ u(x) \geq u(-\delta_0) > 0\ \mathrm{ on }\ (- \delta_0,0).
\end{equation*}
Without loss of generality, we assume $u(x) \geq u(\delta_0) > 0$ on $(0, \delta_0)$.
Furthermore, from \eqref{Condition_CHN} it follows that
\begin{equation*}
- u^\prime(x) \geq \epsilon_0 \frac{u(x)}{x}\ \text{for}\ x > 0.
\end{equation*}
Using this inequality we estimate for $0 < x < \delta_0$:
\begin{align*}
u(x) &\geq u(x) - u(\delta_0) = - \int_x^{\delta_0}u^\prime(s) ds \geq \epsilon_0 \int_x^{\delta_0}  \frac{u(s)}{s} ds \\
& \geq \epsilon_0 u(\delta_0) \int_x^{\delta_0} s^{-1} ds = \epsilon_0 u(\delta_0) \left[ \ln ( \delta_0 ) - \ln (x) \right] \rightarrow \infty\ \text{as}\ x \to 0.
\end{align*}
\par
\medskip
Now we show the claim in higher dimensions. If the single site potential  $U : \RR^d \to [0, \infty)$ does not vanish identically
there is a point $y$ such that $U(y)>0$. Assume without loss of generality that $y$ lies on the $x_1$-axis
and define $u : \RR \to [0, \infty)$ by $u(x_1)=U(x_1,0,\ldots,0)$.
Note that if $U$ satisfies the assumption (F) or (G), respectively then $u$ satisfies (F) or (G) as well and the one-dimensional argument can be applied to $u$. Hence, the statement of the Lemma also holds for $U$.
\end{proof}
In the light of the comments made at the beginning of this section, the occurrence of a singularity is not surprising since in the case of a single-site potential with a polynomial singularity, $u(x) = \lvert x \rvert^{- \alpha}$, we have
\begin{equation*}
u(x/\omega_j) = \left\lvert x/\omega_j \right\rvert^{- \alpha} = \omega_j^\alpha \lvert x \rvert^{- \alpha} = \omega_j^\alpha u(x).
\end{equation*}
and thus the random breathing would correspond to a multiplication which would allow to reduce the breather model to the well-understood alloy type model
$V_\omega(x) = \sum_j \omega_j u(x-j)$.
\ifthenelse{\boolean{journal}}{}{
\section{Proof of Corollary \ref{cor:result1}}\label{sec:proof:cor} 
%
%%%%%%%%%%%%
%%%
%%% Proof of Corollary (scaling) - only in arxiv version
%%%
%%%%%%%%%%%%
%
We fix
\[
\phi =  \sum_{k \in \NN : E_k \leq b} \alpha_k \phi_k
\in \operatorname{Ran} \chi_{(-\infty,b]} (H_{t,L})
\]
and define the map $g : \Lambda_{L/G} \to \Lambda_L$, $g (y) = G \cdot y$.
For all $\phi_k$ the eigenvalue equation reads $ - t \Delta_L \phi_k + V_L \phi_k = E_k \phi_k$ in $\Lambda_L$ where $E_k \leq b$.

We want to transform this into an eigenvalue equation for $\phi_k \circ g$ in $\Lambda_{L/G}$.
Therefore we compose with $g$ and find
\[
 - t ( \Delta_L \phi_k) \circ g + (V_L \circ g) (\phi_k \circ g) = E_k ( \phi_k \circ g)
\]
in $\Lambda_{L/G}$.
The chain rule yields $ ( \Delta_L \phi_k) \circ g = (1 / G^2)  \Delta_{L/G} ( \phi_k \circ g)$ which implies
\[
 - t/G^2 \Delta_{L/G} ( \phi_k \circ g ) + ( V_L \circ g) ( \phi_k \circ g ) = E_k (\phi_k \circ g) .
\]
Thus the eigenvalue equation for $\phi_k \circ g$ is
\[
 -\Delta_{L/G} (\phi_k \circ g) + \left( \frac{G^2}{t} V_L \circ g \right) \left( \phi_k \circ g \right) =  \left( \frac{G^2}{t} E_k  \right) (\phi_k \circ g)
\text{ on }
\Lambda_{L/G}.
\]
Hence,
\[
 \phi \circ g \in \operatorname{Ran} \chi_{(-\infty,G^2 b / t]} \left( -\Delta_{L/G} + (G^2 / t) (V_L \circ g) \right) .
\]
The set $W_\delta(L) \subset \Lambda_L$ arises from a $(G,\delta)$-equidistributed sequence whence the set $W_\delta(L) / G \allowbreak := \{ x \in \RR^d : x \cdot G \in W_\delta(L) \} \subset \Lambda_{L/G}$ arises from a $(1, \delta/G)$-equidistributed sequence.
%
% We summarize our assumptions and observations:
% \begin{itemize}
%  \item $L/G \in \NN$,
%  \item $\delta/G < 1/2$,
%  \item $( G^2 / t) (V_L \circ g): \Lambda_{L/G} \to [- \lVert V \rVert_\infty G^2 /t, \lVert V \rVert_\infty G^2 /t]$ is bounded by $\lVert V \rVert_\infty G^2 / t$,
%  \item $W_\delta(L)/G \subset \Lambda_{L/G}$ arises from a $(1,\delta/G)$-equidistributed sequence,
%  \item $\phi \circ g \in \operatorname{Ran} \chi_{(-\infty,G^2 b / t]} \left( -\Delta_{L/G} + (G^2 / t) (V_L \circ g) \right)$ .
% \end{itemize}
By a coordinate transformation
% \[
%  \lVert \phi \rVert_{W_{\delta}(L)}^2
%  =
%  G^{d} \lVert \phi \circ g \rVert_{W_{\delta}(L) / G}^2
% \quad
%  \text{and}
%  \quad
%  \lVert \phi \rVert_{L^2(\Lambda_{L})}^2
%  =
%  G^{d} \lVert \phi \circ g \rVert_{L^2(\Lambda_{L/G})}^2 .
%  \]
and Theorem~\ref{thm:result1} we obtain
 \[
  \lVert \phi \rVert_{W_\delta(L)}^2
  =
  G^{d} \lVert \phi \circ g \rVert_{W_\delta(L) / G}^2
  \geq
  G^{d} C_\sfuc^{G,t} \lVert \phi \circ g \rVert_{\Lambda_{L/G}}^2
  =
  C_\sfuc^{G,t} \lVert \phi \rVert_{\Lambda_{L}}^2 ,
 \]
where $C_\sfuc^{G,t} = C_\sfuc (d , \delta / G , b G^2 / t , \lVert V \rVert_\infty G^2 / t)$.
}
\small

\section*{Acknowledgement}

This work has been partially supported by the DFG under grant \emph{Unique continuation principles and equidistribution properties of eigenfunctions}.
and by the binational German-Croatian DAAD-MZOS project \emph{Scale-uniform controllability of partial differential equations}.
I.N. was partially supported by HRZZ project grant 9345. M.T.\ thanks Constanza Rojas-Molina for pointing out that the initial length scale estimate follows  from the unique continuation principle.

%
%\bibliographystyle{amsalpha}
% \bibliographystyle{amsplain}
%\bibliography{lit_DAAD}%
\providecommand{\bysame}{\leavevmode\hbox to3em{\hrulefill}\thinspace}
\providecommand{\MR}{\relax\ifhmode\unskip\space\fi MR }
% \MRhref is called by the amsart/book/proc definition of \MR.
\providecommand{\MRhref}[2]{%
  \href{http://www.ams.org/mathscinet-getitem?mr=#1}{#2}
}
\providecommand{\href}[2]{#2}

\end{document}